\documentclass{article}
\usepackage[left=3cm,right=3cm,top=2.5cm,bottom=2.5cm,includeheadfoot]{geometry}
\usepackage{etoolbox}
\usepackage{graphicx}
\usepackage{amsmath,amssymb,amsthm}
\usepackage{enumitem}
\usepackage{dsfont}
\usepackage{framed}
\usepackage{wrapfig}
\usepackage{booktabs}

\def\cf{\emph{cf.}} 
\newcommand{\ie}{\textit{i.e.}}

\usepackage[utf8]{inputenc}

\usepackage{tabularx}

\usepackage{times}
\usepackage{color}

\usepackage{graphpap}
\usepackage{nicefrac}
\usepackage{MnSymbol}
\usepackage{mathtools}
\usepackage{multirow}
\usepackage{bm}
\usepackage{hyperref}
\usepackage[capitalize]{cleveref}
\usepackage{placeins}
\usepackage{tikz}
\usepackage{tikz-3dplot}
\usepackage{pgfplots}
\usepackage{todonotes}
\usepackage{xparse}

\usetikzlibrary{calc,math}
\newcommand*\mycirc[1]{%
  \begin{tikzpicture}
    \node[draw,circle,inner sep=1pt] {#1};
  \end{tikzpicture}}
\definecolor{color0}{rgb}{0.95,0.95,0.95} % plot background color
\definecolor{cb-pink}{HTML}{cc78bc}
\definecolor{cb-darkblue}{HTML}{0173b2}
\definecolor{cb-vividorange}{HTML}{de8f05}
\definecolor{cb-green}{HTML}{029e73}
\definecolor{cb-strongorange}{HTML}{d55e00}
\definecolor{darkgray}{HTML}{404040}

\numberwithin{equation}{section}

\newcommand{\R}{\mathbb{R}}
\newcommand{\N}{\mathbb{N}}
\newcommand{\PP}{\mathbb{P}}
\newcommand{\parametrization}{\Pi}
\newcommand{\eps}{\epsilon}
\newcommand{\opd}{\,\mathrm{d}}

\newcommand{\discGrad}[2]{\partial_{(#1,#2)}}
\newcommand{\discHess}[2]{\partial^2_{(#1,#2)}}

\newcommand{\energy}{\mathcal E}

\newcommand{\err}{\mathrm{err}}

\newcommand{\distortion}{{\gamma}}
\newcommand{\Distortion}{{\Gamma}}
\newcommand{\densitymm}{\rho^\eps_{M\times M}}
\newcommand{\NNSpace}{\mathcal F^\eps}
\newcommand{\phisub}{\phi_V}
\newcommand{\embedman}{M_0}
\DeclarePairedDelimiter\abs{\lvert}{\rvert}
\DeclarePairedDelimiter\norm{\lVert}{\rVert}

\DeclarePairedDelimiter\absav{\lvert}{\rvert_{\text{av}}}

\DeclareMathOperator{\av}{av}
\DeclareMathOperator{\Grad}{grad}

\DeclareMathOperator{\Hess}{Hess}

\DeclareMathOperator{\tr}{tr}

\DeclareMathOperator{\PT}{PT}
\DeclareMathOperator{\cone}{\mathcal{C}}
\DeclareMathOperator{\discprod}{\mathcal{D}}

\DeclareMathOperator{\lrelu}{LeakyReLU}

\ProvideDocumentCommand\fracp{O{}m}{\frac{\partial #1}{\partial #2}}
\ProvideDocumentCommand\fracpsq{O{}mm}{\frac{\partial^2 #1}{\partial #2 \partial #3}}
\ProvideDocumentCommand\sample{O{\epsilon}}{\mathcal S_{#1}}

\makeatletter
\newcommand{\labeltext}[2]{%
  \@bsphack
  \csname phantomsection\endcsname % in case hyperref is used
  \def\@currentlabel{#1}{\label{#2}}%
  \@esphack
}
\makeatother

\MakeRobust{\ref}% avoid expanding it when in a textual label

\newtheorem{remark}{Remark}
\newtheorem{lemma}{Lemma}
\newtheorem{proposition}{Proposition}
\newtheorem{theorem}{Theorem}
\newtheorem{corollary}{Corollary}

\begin{document}
%%-----------------------------
%%      the top matter
%%-----------------------------
\title{Convergent autoencoder approximation of
  low bending and low distortion manifold embeddings} %\thanks{\dots}\thanks{...}% At most 5 thanks
\author{Juliane Braunsmann\thanks{University of M\"unster, M\"unster , Germany 
  (\url{j.braunsmann@uni-muenster.de}, 
  \url{benedikt.wirth@uni-muenster.de}).}
\and Marko Rajkovi\'c\thanks{Institute for Numerical Simulation, University of Bonn, Bonn, Germany 
  (\url{marko.rajkovic@ins.uni-bonn.de}, 
  \url{martin.rumpf@ins.uni-bonn.de}).}
\and Martin Rumpf\footnotemark[3]
\and Benedikt Wirth\footnotemark[2] }

\maketitle

%%----------------------------------
%%          begin of paper
%%----------------------------------

\begin{abstract} 
Autoencoders are widely used in machine learning for 
dimension reduction of high-dimensional data.
The encoder embeds the input data manifold into a lower-dimensional latent space, while the decoder represents the inverse map,
providing a parametrization of the data manifold by the manifold in latent space.
We propose and analyze a novel regularization for learning the encoder component of an autoencoder:
a loss functional that prefers isometric, extrinsically flat embeddings and allows to train the encoder on its own.
To perform the training, it is assumed that the local Riemannian distance and the local Riemannian average can be evaluated for pairs of nearby points on the input manifold.
The loss functional is computed  via  Monte Carlo integration. Our main theorem identifies a geometric loss functional 
of the embedding map as the $\Gamma$-limit of the sampling-dependent loss functionals.
Numerical tests, using image data that encodes different explicitly given data manifolds, show that smooth manifold embeddings into latent space are obtained.
Furthermore, due to the promotion of extrinsic flatness, interpolation between not too distant points on the manifold
is well approximated by linear interpolation in latent space.
 \end{abstract}
%%-----------------------------
%%       text
%%-----------------------------
\section{Introduction}\label{sec:intro}
A central machine learning task is to identify high-dimensional objects from given data sets with points in a 
(suitably generated) latent manifold. Such a representation is frequently defined via an \emph{encoder map} 
of an \emph{autoencoder} acting on input objects and mapping them into a low-dimensional Euclidean \emph{latent space}. The 
\emph{latent manifold}
is then the image of this encoder map. 
The associated \emph{decoder} maps back points in latent space to points in the input space. 
The assumption that an observed high-dimensional data set actually forms a low-dimensional manifold -- the image of the latent manifold under the decoder map (which is often called \emph{hidden manifold} due to its a priori unknown topology and geometry) -- is called the \emph{manifold hypothesis}. 
Encoder and decoder can be trained as {deep} neural networks via the minimization of a functional called \emph{reconstruction loss} which compares the input data with its image under the composition of the encoder and decoder mapping.
This approach is an instance of \emph{deep manifold learning}.
One aims for a latent manifold whose geometry is close to the hidden manifold, however,
measuring only a reconstruction loss is not enough to accurately recover the geometry present in the data.
Different strategies to favor smoothness of the encoder and decoder mapping and 
thus regularity of the latent manifold have been investigated.
Among them are methods which promote sparsity~\cite{RaPoChLe07}, contractive autoencoders~\cite{RiViMuGlBe11}  
which use loss functionals penalizing the norm of the Jacobian of the encoder mapping, denoising autoencoders~\cite{ViLaBeMa08}, or variational autoencoders~\cite{KiWe13} which regularize the latent representation to match a tractable probability distribution.

A major deficit of autoencoders is that they frequently fail to reproduce the
{statistical distribution of input data}
 in the latent space. 
To resolve this issue, approaches promoting \emph{low distortion} or \emph{isometric} encoder maps are proposed.
In \cite{SchTa19}, a loss functional measures the difference between distances in the pushforward metric and distances in latent space in order to quantify the lack of isometry.
The loss functional from  \cite{PaTaBrKi19} 
compares Euclidean distances in latent space with geodesic distances on the input manifold.
In \cite{KaZhSaNa20}, {isometric}, \ie~ length-preserving, encoder maps are used to more accurately pushforward distributions from input to latent space. A loss function based on Shannon-Rate-Distortion is used for this purpose.
The loss functional from \cite{AtGrLi20} promotes isometry of the decoder map (by penalizing deviation from a non-orthogonal Jacobian matrix), and that the 
encoder is a pseudo-inverse of the decoder (by enforcing the Jacobians of decoder and encoder to be transposes of each other).
In \cite{PeLiDi+20}, isometric embeddings in latent space are learned to obtain standardized data coordinates from scientific measurements. To this end, the Jacobian is approximated via normally distributed sampling around each data point (so-called \emph{bursts}), and the deviation of the local covariance of bursts from the identity is used to measure the lack of orthogonality of the Jacobian.
In \cite{YoYoSoPa21}, these approaches are extended by studying autoencoders which approximately preserve  not only distances, but also angles and areas. 
This is achieved by a loss functional depending on eigenvalues of the pullback metric defined in terms of Jacobians of encoder and decoder.
In \cite{KoClMi21}, the spectral method of Laplacian Eigenmaps~\cite{BeNi01} was combined with a loss which 
penalizes, in terms of Lipschitz constants, the deviation of the embedding map and its inverse from the identity.

Data manifolds often come with a metric 
encoding the cost of local variations on the manifold. 
For sufficiently regular data manifolds, it is shown in \cite{ShKuFl18} how to transfer this metric to the latent manifold,
thereby turning it into a \emph{Riemannian} manifold. 
This in principle allows to compute shortest paths, 
exponential maps, and parallel transport in latent space, which the decoder can pushforward
to the data manifold. Latent spaces of variational autoencoders were first studied as a Riemannian manifold in \cite{ChKlKu+18}.
A desired property of autoencoders is to enable meaningful interpolation between data points via an affine interpolation in the latent space. 
In \cite{BeRaRoGo18} an \emph{adversarial regularizer} is proposed to ensure visually realistic interpolations in latent space.
To this end, the adversarial regularizer tries to make the decoding of interpolations in latent space indistinguishable from real data points.

In this paper we study embeddings from general input manifolds of high-dimensional data into low-dimensi\-onal latent spaces.
We propose a loss functional based on randomly sampled point pairs on the input manifold.
We assume that the Riemannian distance on the input manifold and the Riemannian average
can be computed for each such pair of points. The loss functional consists of two components. The first component measures the distortion of the distance, thus promoting approximately isometric latent space embeddings. The second component penalizes large bending, thus promoting flatness. 
Altogether, we obtain a low bending and low distortion (LBD) discrete loss functional. 
It depends on a parameter $\epsilon$, defined as the maximal distance of each point pair sampled
 on the input manifold.
As the number of samples increases, we obtain a 
Monte Carlo limit of the initially discrete loss functional as a nonlocal loss functional, a double integral over the pairs of input manifold points which are at most $\epsilon$ apart.
The main theorem of this paper identifies the $\Gamma$-limit of these nonlocal Monte Carlo limits, for sampling distance 
$\epsilon\to 0$, as a local
geometric loss functional of embeddings.
This loss functional consists of a distortion and a bending energy that are proper limits of their discrete counterparts.

From an abstract numerical analysis point of view,
the above limit problem belongs to the class of problems in which an infinite-dimensional optimization problem is approximated in two steps: 
First, the model or objective functional is approximated introducing an auxiliary length scale parameter $\eps$.
Then the new objective functional is restricted to a space or set of discretized functions, parametrized by finitely many parameters that can be optimized for.
To relate the final discrete problem to the original one, 
both the limit $\eps\to0$ and the limit of increasing expressivity of the discrete function space have to be tackled. At the same time, the parameter $\eps$ and the chosen discrete function space have to be compatible  with each other.
For instance, in the well-known example 
of finite element discretizations of phase field approximations to sharp interface problems, the grid size needs to decrease faster than the phase field parameter $\eps$ in order to avoid discretization artifacts (\cf~\cite{FePr01}).
In our setting, the discrete function space consists of {realizations of} deep neural networks,
and, for our type of loss functional, it becomes necessary that these functions are sufficiently regular with respect to $\eps$,
\ie, that $\eps$ decreases faster than some regularity measure of the discretized functions.
We achieve this by imposing easily implementable bounds on the network weights which become weaker for decreasing $\eps$.

In the presented numerical tests we use
digital image data that represents different explicitly given hidden data manifolds.
Thus, the dimension of our input space equals the number of pixels times the number of colors. 
\begin{figure}
 \centering
 \input{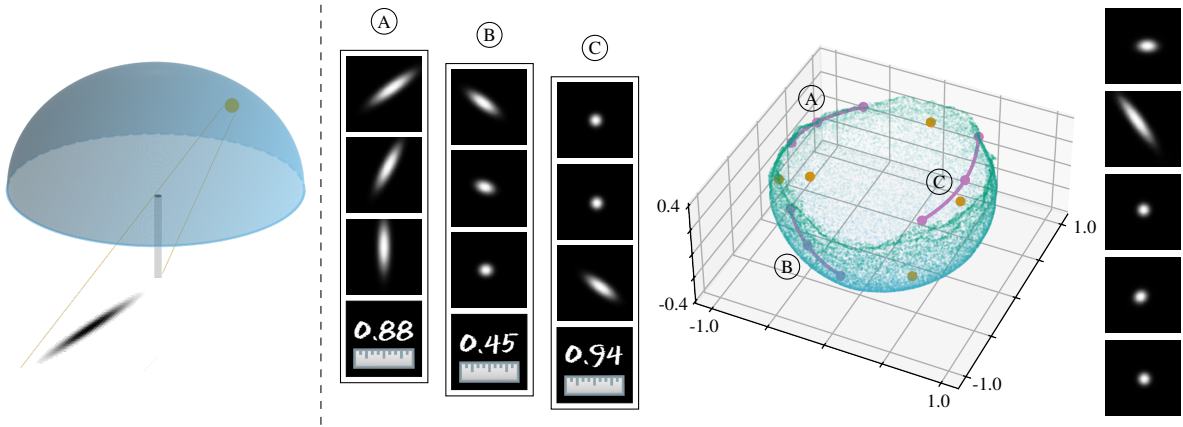}
 \caption{Results for dataset \ref{dataset:s} (\cf~\cref{subsec:setup}) for $\lambda=0$ and $\eps=\pi/8$ (with maximal distance being $\pi$). From left to right:
   a sketch of the sundial configuration; training data (image pairs with their geodesic average and distance);
   the latent manifold $\phi(M)\subset \R^{16}$ projected into $\R^3$ via PCA;
  decoder outputs for the orange points in latent space. 
  The encoder was trained separately from the decoder and the training was stopped after the value of the functional $E^{\sample}[\phi]$ evaluated on a test set did not decrease for 20 epochs. Pairs with distance below $\frac{1}{100}$ of the maximal distance were rejected. The decoder was trained until an accuracy of $10^{-5}$ was reached on the same test set. The percentage of explained variance for the first three components lies at 97.8\%, with a threshold of 99\% reached at 6 components.}
  \label{fig:masterpiece}
 \end{figure}
\Cref{fig:masterpiece} shows an example with input images representing shadows of a sundial.
The learned embedded manifold is illustrated by randomly sampling the input data
and visualizing the corresponding output of the trained encoder network.
For this visualization, the full latent space was projected onto the three dominant PCA directions of the obtained point cloud in latent space.
The input manifold consists of shadow images for all possible positions of a lightsource on a hemisphere. The training data consists of tuples of two such images, their distance and their mean, where the Riemannian distance and mean are induced by the geometry of the hemisphere.
The figure demonstrates that our
LBD manifold embedding clearly reflects the geometric characteristics of the original hidden manifold, the hemisphere.
In contrast to other isometry promoting approaches  the \emph{Jacobian of the encoder is not approximated}
(via backpropagation or complete correlation of all nearest neighbors), but a simple 
\emph{Monte\,Carlo  sampling of point pairs} is used.
Our numerical experiments confirm that the bending loss 
significantly increases \emph{smoothness} of the resulting latent space manifold when compared to a pure isometry loss. Indeed, 
solely using an isometry loss functional allows quite irregular Nash--Kuiper embeddings. 
Furthermore, the decoder maps 
linear interpolations in latent space to reasonable interpolations on the data manifold.

The underlying loss functional has been proposed already in the conference paper \cite{BrRaRu21},
which contains a short derivation of our embedding ansatz and motivates the loss functional. 
The only significant theoretical result was a consistency result using Taylor expansion arguments,  
under the implicit assumption that the manifold embedding is already known to be sufficiently smooth. 
This paper significantly extends the results in the following directions: 

We show existence of minimizers for the nonlocal loss functional in a specific
class of functions which consists of realizations of smooth neural networks satisfying (readily enforceable) bounds in certain norms.

We discuss the dependence of the limit functional on the used sampling strategy. 

We prove the variational convergence (more specifically, Mosco-convergence as a variant of  $\Gamma$-convergence)
of the nonlocal autoencoder loss functionals to a continuous second order deformation functional.
As a consequence, minimizers of the autoencoder loss functionals (which are  {realizations of} 
 neural networks satisfying the above-mentioned norm bounds) converge to a minimizer of the limit functional.
This moreover automatically implies the existence of minimizers of the latter.

We present novel and systematic numerical experiments. A substantial part of the result section is devoted to a thorough quantitative analysis of the experiments presented in our conference paper 
and a detailed study of an additional example with the Klein bottle as the underlying hidden manifold geometry.

The paper is structured as follows.
\Cref{sec:discreteloss} introduces the regularization loss as the Monte Carlo limit for dense sampling of input data.
In \cref{sec:analysis} we derive the $\Gamma$-convergence result for the regularization loss.
We start with a recapitulation of some facts about the function spaces necessary in our study in \cref{subsec:function_spaces}. In \cref{subsec:discrete_existence},we prove the existence of a minimizer for the {nonlocal} regularization functional under suitable assumptions on the class of embedding networks. FInally, in \cref{subsec:limit}  we demonstrate the Mosco-convergence of the discrete functionals to the continuous limit functional.
In \cref{sec:experiments} we present some numerical experiments.
As a proof of concept, we train autoencoders on several image datasets representing a priori known data manifolds. We describe the datasets in \cref{subsec:setup} and the autoencoder setup in \cref{subsec:autoencoder}. Finally, we discuss the numerical results in \cref{subsec:visualization,subsec:interpolation,subsec:noise}.

%%%%%%%%%%%%%%%%%%%%%%%%%%%%%%%%%%%%%%%%%%%%%%%%%%%%%%%%%%%%%%%%%%%%%%%%%%%

\section{A low bending and low distortion regularization for encoders}
\label{sec:discreteloss}
We consider a compact, $m$-dimensional Riemannian manifold $\overline M$ with metric $g$, with or without boundary, embedded in some very high-dimensional space $\R^n$.  In our applications $n=N$ for grayscale or $n=3N$ for RGB images
with $N$ pixels. 
Our aim is to compute an embedding $\phi$ of $\overline M$ into $\R^l$,
where we think of the dimension $l$ as being only moderately larger than $m$
(e.g. $l=2m$ so that the existence of a smooth embedding is guaranteed by Whitney's embedding theorem~\cite{Wh92}). We call $\R^l$ the \emph{latent space} and $\phi(\overline M)\subset \R^l$ the \emph{latent manifold}.
Formulated in terms of autoencoders this amounts to learning 
a pair of maps 
\begin{equation*} 
\phi=\phi_\theta:\overline M\to\R^l,\text{ }  \psi_\xi:\R^l\to\R^n \text{ with }  \psi_\xi(\phi_\theta(x))\approx x \text{ for all } x\in\overline M.
\end{equation*}
The \emph{encoder} $\phi_\theta$ and \emph{decoder} $\psi_\xi$ are implemented as deep neural networks with 
vectors of encoder and decoder parameters $\theta$ and $\xi$. The encoder $\phi_\theta$ is actually defined on the embedding space $\R^n$, but we don't distinguish between $\phi_\theta$ and $\phi_\theta\vert_{\overline{M}}$, since we are not interested in $\phi_\theta$ outside of $\overline{M}$.
The parameters are typically optimized via the minimization of a loss function measuring the difference between the original points $x$ and their reconstructions $\psi_\xi(\phi_\theta(x))$.
An appropriate structure and regularity of the embedding into latent space can be promoted by a suitable, geometrically inspired regularization loss for the encoder. This is of particular interest for downstream tasks such as classification~\cite{AtGrLi20}, Riemannian interpolation and extrapolation~\cite{BeRaRoGo18}, clustering or anomaly detection~\cite{ZhLiCh+21}.
Indeed, if $\overline M$ were isometric to $\R^m$ and thus could be embedded isometrically into an $m$-dimensional affine subspace of $\R^l$,
then, for instance, the most important and basic operations of computing distances and interpolations would become trivial in latent space. 
The same holds for many other downstream tasks.
Although such an embedding is usually prevented by the intrinsic or the global geometry of $\overline M$,
the hope for simplified downstream tasks motivates the search for an embedding as close as possible to isometric and flat, at least locally. 
To this end, we suggest the following two simple objectives for any two not too distant points $x,y\in\overline M$:
\smallskip

%%%
\noindent [{\bf I}sometric] The intrinsic Riemannian distance between $x$ and $y$ in $\overline M$ 
should differ as little as possible from the Euclidean distance between the 
latent codes $\phi(x)$\,and\,$\phi(y)$.\smallskip

%%%
\noindent [{\bf F}lat]
The encoder image of a (weighted) average between $x$ and $y$ in $\overline M$ 
should deviate as little as possible from the (weighted) Euclidean average between $\phi(x)$ and $\phi(y)$.
\smallskip
%%%

Finding an \emph{isometric embedding},
as extensively pursued in the literature (\cf~\cref{sec:intro}),
essentially is equivalent to ensuring the first objective (\textbf{I}) for infinitesimally close points $x,y\in\overline M$.
From the mathematical point of view, though, isometry by itself is insufficient to ensure regular embeddings
since the family of isometric embeddings is very large and contains quite irregular elements. In particular,
Nash--Kuiper embeddings are in general only H\"older differentiable \cite{Ku55a,Ku55b,Na54}.
Therefore, we suggest (\textbf{I}) and (\textbf{F}) as stronger objectives:

$\bullet$ The isometry {or low distortion} objective (\textbf{I}) asks that the \emph{intrinsic} distances between $x$ and $y$ in $\overline M$ are approximated by the \emph{extrinsic} distances between $\phi(x),\phi(y)\in\R^l$ in latent space (rather than the intrinsic distances in the latent manifold $\phi(\overline M)$ with the metric induced by $\R^l$, which would define an isometric embedding).

$\bullet$ The flatness or {low} bending objective (\textbf{F}) enforces some second order low bending regularity or flatness on $\phi$ by requiring that the \emph{geodesic interpolation} between $x$ and $y$ in $\overline M$ is well approximated by extrinsic \emph{linear interpolation} in the latent space $\R^l$.

\subsection{A low bending and low distortion loss functional}
In what follows we give more details on the manifold $\overline M$ and the LBD loss functional.

Let $(M_0,g)$ be a smooth $m$-dimensional Riemannian manifold without boundary. Let $M$ be an open subset of $M_0$ so that $(M, g)$ is a smooth $m$-dimensional Riemannian submanifold of $\embedman$  and let $\overline M$ be its closure. %In particular, $M=\interior{\overline M}$.
Further, let $\overline M$ be compact and such that $\overline M$ is a (not necessarily smooth) $m$-dimensional manifold, potentially with boundary, in the following sense:
For each $x\in\overline M$ there exist a (relatively) open neighborhood $U$ of $x$ in $\overline M$ and a homeomorphism $\phi\colon U\to V$ with $V=B_1(0)\coloneqq\{y\in\R^m:|y|<1\}$ for $x\in M$ and $V=B_1(0)\cap H^m$ for $x\in\overline M\setminus M$, where $$H^m\coloneqq\{(y_1, \dots, y_m)\in \R^m \,:\, y_m \geq 0\}$$ is the closed $m$-dimensional upper half-space.

Let us denote by $d_M(x,y)$ the Riemannian distance between any two points $x,y\in M$.
Further, let $TM$ denote the tangent bundle of $M$, $T_xM$ the tangent space to $M$ at $x\in M$ and $\exp_x$ the Riemannian exponential map defined on (a subset of) $T_xM$.
As mentioned before, for training of our encoder we will employ the Riemannian mean of point pairs $(x,y)\in M\times M$. By Riemannian mean we mean the midpoint of the unique geodesic connecting $x$ and $y$ in $M$, if it exists.
If $y=\exp_xv$, where $v\in T_xM$ is the initial velocity of this unique shortest geodesic connecting $x$ with $y$,
the Riemannian mean is given by $\av_M(x,y)=\exp_x\frac v2$.
If no such $v\in T_xM$ exists (for instance due to nonunique shortest geodesics), the Riemannian mean is not defined.
(Note that one could slightly generalize the definition of the Riemannian mean as the midpoint of shortest connecting curves in $\overline M$,
which may include curves touching the boundary of $M$. However, this way there might exist several $y\in M$ with same $\av_M(x,y)$, for which reason we refrain from this generalization.)
Therefore we will require some natural conditions on the behavior of the Riemannian exponential map associated with $M$.
To state those, let $V_x\subset T_xM$ denote the (automatically star-shaped) largest open set on which $\exp_x$ is defined.
We will assume that there exist constants $r_0>0$ and 
$0 < \kappa < \pi/2$ such that the following holds:

\newlength{\mylongest}
\setlength{\mylongest}{\widthof{(M2$'$)}}
\addtolength{\mylongest}{\labelsep}
\SetLabelAlign{CenterWithParen}{\makebox[\mylongest]{#1}}

\begin{enumerate}[align=CenterWithParen,labelwidth=\mylongest,leftmargin=!]
\item[\labeltext{(M1)}{M1}(M1)]
(injectivity condition)
For $x\in M$ let $B_{r_0}^{T_xM}(0)\coloneqq\{v\in T_xM\,:\,g_x(v,v)<r_0^2\}$ denote the open ball in tangent space of radius $r_0$.
Then the Riemannian exponential $\exp_x$ is injective on $V_x\cap B_{r_0}^{T_xM}(0)$ for all $x\in M$.
\item[\labeltext{(M2$'$)}{Lipschitz}(M2$'$)] (strong Lipschitz condition) $\overline M$ is strongly Lipschitz, meaning for each $x\in \partial M$ there exists an open neighborhood $U$ of $x$ in $\embedman$, a smooth local chart $h\colon U \to \R^m$ of the manifold $\embedman$ with $h(x)=0$, $\gamma\colon \R^{m-1}\to \R$ a Lipschitz function with $\gamma(0)=0$ and $\delta > 0$ such that 
$$h(U\cap M) = \{(x', \gamma(x')+t)\,:\, 0 < t < \delta, x'\in \R^{m-1}, \abs{x'}<\delta \}.$$%
\end{enumerate}
Note that \ref{M1} is for instance fulfilled when $r_0$ is chosen as the injectivity radius of the embedding manifold $M_0$.
Both above conditions imply the following cone condition:
\begin{enumerate}[align=CenterWithParen,labelwidth=\mylongest,leftmargin=!]
  \item[\labeltext{(M2)}{M2}(M2)] 
(cone condition)
There exists a measurable map $\iota\colon M\times \R^m \to TM$ such that, setting $\iota_x(v)=\iota(x,v)$ for ${x\in M}, {v\in \R^m}$, we have that $\iota_x$ is an isometric isomorphism between $\R^m$ and $T_xM$, 
and $V_x$ contains $\iota_x(\cone_{r_0,\kappa})$, where $\cone_{r_0,\kappa}$ is the cone
\begin{equation*}
  \cone_{r_0, \kappa} \coloneqq \left\{w \in \R^m \,:\, 
  0 \leq \abs w < r_0, \,\vert w\cdot e_1 \vert \geq \cos(\kappa) \vert w \vert \right\}
\end{equation*}
of height $r_0$ and maximum angle $\kappa$ with the first standard Euclidean basis vector $e_1\in\R^m$. The dot $\cdot$ denotes the standard inner product in $\R^m$.
\end{enumerate}
The definition of Lipschitz domains in manifolds follows \cite{MiTa99}. In the Euclidean setting, it is well known that a Lipschitz domain fulfills the cone condition, see e.g.\ \cite{AdFo03}. This similarly holds in our setting, which can be proven in a similar fashion as in \cite{MiTa99}.
For manifolds without boundary, the largest $r_0$ coincides with the usual injectivity radius,
and we can replace the cone $\cone_{r_0, \kappa}$ by the open ball $B_{r_0}^m(0)\subset\R^m$ centered at the origin with radius $r_0$.
For manifolds with boundary, the cone condition is not automatically satisfied: Consider for example the submanifold 
$$\overline M = \left\{(x,y)\in \R^2\,:\,y \geq \sqrt{\abs x}, y \leq 1\right\}$$
of $\R^2$ with the inherited metric. Then no cone can be positioned at $(0,0)$. 

Using the isometric linear mapping {$\iota_x \colon \R^m \to T_xM$} from \ref{M2} and setting {$U_x\coloneqq\iota_x^{-1}(V_x)\cap B_{r_0}^m(0)$},
we can define a local parametrization of $M$ by {$\parametrization_x \colon U_x\to M$, $\parametrization_x\coloneqq\exp_x\circ\iota_x$.}
This parametrization is known as \emph{normal coordinates around $x$} (\cf~\cref{fig:cone}).
For later notational convenience, we extend $\parametrization_x$ measurably (but arbitrarily) beyond $U_x$.
Let us further introduce the notations $U_x^\eps\coloneqq U_x\cap B_\eps^m(0)$ and $D_\eps^M(x)\coloneqq \parametrization_x(U_x^\eps)$ for $0<\eps<r_0$.
The neighborhood $D_\eps^M(x)$ of $x$ consists of all points in the $\eps$-neighborhood of $x$ that can be reached from $x$ via the Riemannian exponential map
(it coincides with the $\eps$-neighborhood of $x$ for geodesically convex manifolds $M$ or for points $x$ with Riemannian distance to the boundary larger than $\eps$).
Thus,  condition \ref{M1} ensures that the Riemannian mean of $x$ and any $y\in D_\eps^M(x)$ is well-defined.
\begin{figure}
\begin{center}
\setlength{\unitlength}{0.03\linewidth}%
\begin{picture}(20,9)
\put(0,0){\includegraphics[width=20\unitlength]{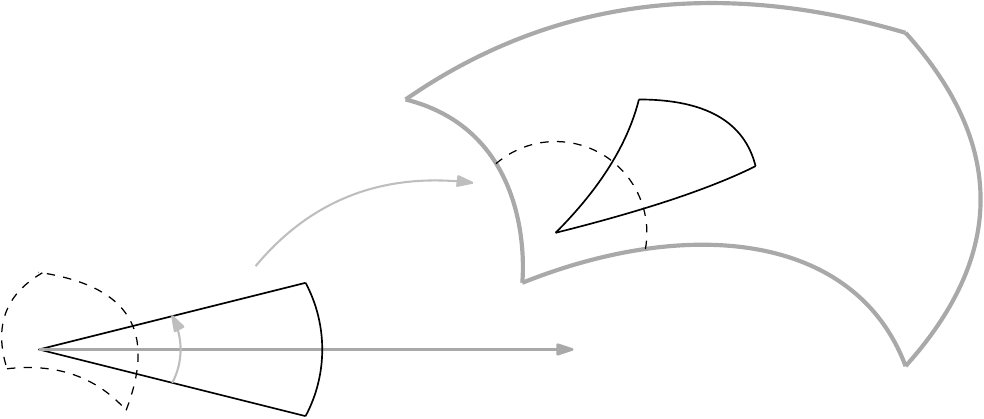}}
\put(-0.4,1){\small$0$}
\put(18,4){\small$M$}
\put(15,2.5){\small$\partial \overline M$}
\put(11,0.85){\small$e_1$}
\put(6.7,0.85){\small$r_0 e_1$}
\put(6.56,1.40){\circle*{0.25}}
\put(2.775,2.4){\small$2\kappa$}
\put(4.3,0.75){\small$\cone_{r_0, \kappa}$}
\put(6,4.9){\small$\parametrization_x$}
\put(11,3.35){\small$x$}
\put(0.3,3.5){\small$U_x^\eps$}
\put(9.6,6){\small$D_\eps^M(x)$}
\end{picture}
\end{center}
\caption{The parametrization $\parametrization_x$ that maps the cone $\cone_{r_0,\kappa}\subset U_x \subset \R^m$ into $M$ and the neighborhoods $U_x^\eps\subset U_x$ onto $D_\eps^M(x)$.}
\label{fig:cone}
\end{figure}
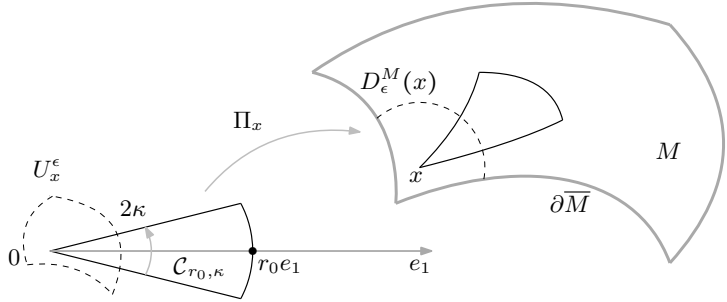

We can now define the loss functional and the corresponding sampling framework.
Let 
\begin{equation*}
  \discprod_\eps \coloneqq \{(x,y)\in M\times M\,:\,y\in D_\eps^M(x)\}.
\end{equation*}
For a fixed \emph{sampling radius} $\eps\in(0,r_0)$, we consider
a finite set $\sample\subset \discprod_\eps$
of discrete samples as training data for the encoder. We then define the \emph{discrete sampling loss functional}
\begin{equation}\label{eq:sampling_energy}
E^{\sample}[\phi] \coloneqq
\frac1{|\sample|}\sum_{(x,y)\in\sample}\left(\distortion(|\discGrad{x}{y}\phi|) + \lambda \; |\discHess{x}{y}\phi|^2\right),
\end{equation}
with $\lambda>0$. The first and second order difference quotients are defined as 
\begin{equation*}
\discGrad{x}{y}\phi\coloneqq\frac{\phi(y)-\phi(x)}{d_M(x,y)}, ~~ \discHess{x}{y}\phi\coloneqq 8\frac{\av_{\R^l}(\phi(x), \phi(y))-\phi(\av_M(x,y))}{d_M(x,y)^2},
\end{equation*}
where $\av_{\R^l}(a,b)=(a+b)/2$ denotes the linear average in $\R^l$  
and $\distortion:[0,\infty)\to[0,\infty)$ a function with a unique minimum $\distortion(1)=0$.
Then,  the first term in $E^{\sample}$ has
a strict minimum for  $|\discGrad{x}{y}\phi|=1$. This promotes $\abs{\phi(x)-\phi(y)}\approx d_M(x,y)$ and thus prefers 
an approximate isometry with low distortion.
{An example of such a function, used in the rest of this paper,  is  $$\distortion(s) \coloneqq s^2+\frac{(1+c^2)^2}{s^{2}+c^2}-2-c^2,~~ c>0. $$}
Let us emphasize that $\distortion(s)$ is chosen such that it is finite in $s=0$ and that the composition $x\mapsto \distortion(\norm x)$, for $x \in \R^l$, is smooth.
The second term in $E^{\sample}$ penalizes the deviation of intrinsic averages on $\phi(M)$ from extrinsic ones in $\R^l$.
Note that this does not only penalize bending or any extrinsic curvature of $\phi(M)$ in $\R^l$,
but in addition it also penalizes deviation of the inplane parametrization of $\phi(M)$ from a linear one.
One further notices that the functional is rigid motion invariant by construction, \ie, composition of $\phi$ with a rigid motion does not change the energy.
\subsection{Monte Carlo limit for dense sampling}\label{subsec:MonteCarlo}
We next discuss the Monte Carlo limit of the sampling loss functional \eqref{eq:sampling_energy} as the number of samples tends to infinity.
This limit depends on the sampling strategy.
We consider a pair of random variables $(X,Y)$ taking values in $\discprod_\eps$, with a distribution that will be defined below. Let $\sample[n]$ be a sequence of random finite sets with $\sample[n]\subset \discprod_\eps, $ each containing $n$ independent samples of $(X,Y)$.  By the strong law of large numbers~\cite[Chapter 17.3]{Lo77}, we then have almost sure convergence to a \emph{continuous sampling loss functional}
\begin{equation*}
\energy^\eps[\phi] \coloneqq \mathbb{E}\left[\distortion(|\discGrad{X}{Y}\phi|) + \lambda \; |\discHess{X}{Y}\phi|^2 \right]
=\lim_{n\to\infty} E^{\sample[n]}[\phi].
\end{equation*}
Let us consider three exemplary ways of sampling, leading to three variants $\energy^\eps_i$, $i=1,2,3$ of $\energy^\eps$:
\begin{enumerate}[label=(S\arabic*)]
\item \label{S1} Sampling $x \in M$ uniformly and then choosing $y\in D_\eps^M(x)$ uniformly.
Thus, denoting by $V_g$ the Riemann--Lebesgue volume measure on $M$ (\cf~\cite[Sec. 1.2]{Hebey96}),
we have 
$\PP[X\in A] = V_g(A)/V_g(M)$ and $Y$ is given via the conditional law $$\PP[Y\in A|X=x]=\frac{V_g(A\cap D_\eps^M(x))}{V_g(D_\eps^M(x))}$$ for any measurable $A\subset M$.
The limit is
\begin{equation}\label{eq:MonteCarlo1}
   \energy^\eps_1[\phi] \coloneqq \strokedint_M\!\! \strokedint_{D_\eps^M(x)}  \distortion(|\discGrad{x}{y}\phi|)+ \lambda \, |\discHess{x}{y}\phi|^2 \opd V_g(y)\opd V_g(x).
\end{equation}
\item \label{S2} Sampling $x \in M$ uniformly and then choosing $v \in U_x^\eps$  uniformly
 and defining $y=\parametrization_x(v)$. 
We have $\PP[X\in A]= V_g(A)/{V_g(M)}$ as above, while $Y$ is given via the conditional law as 
$$\PP[Y \in A| X=x] = \frac{\mathcal L(\parametrization_x^{-1}(A)\cap U_x^\eps)}{\mathcal L(U_x^\eps)},$$
for any measurable $A \subset M$. Here, 
 $\mathcal L$ denotes the $m$-dimensional Lebesgue measure.
This gives
\begin{equation}\label{eq:MonteCarlo2}
   \energy^\eps_2[\phi]\coloneqq\strokedint_M\!\! \strokedint_{U_x^\eps} \distortion(|\discGrad{x}{\parametrization_x(v)}\phi|) + \lambda \, |\discHess{x}{\parametrization_x(v)}\phi|^2 \opd v \opd V_g(x),
\end{equation}
where $\opd v$ indicates the integration with respect to the Lebesgue measure $\mathcal L$.
\item \label{S3} Sampling pairs $(x,y) \in \discprod_\eps$ uniformly. 
We have $$\PP[(X,Y) \in A]=\frac{(V_g\otimes V_g)(A\cap \discprod_\eps)}{(V_g\otimes V_g)(\discprod_\eps)}$$
for any measurable $A \subset M\times M,$ so that
\begin{equation}\label{eq:MonteCarlo3}
   \energy^\eps_3[\phi]\coloneqq \dfrac{1}{\int_M \!\! V_g(D_\eps^M(x)) \opd V_g(x)}\int_M \!\! \int_{D_\eps^M\!(x)} \distortion(|\discGrad{x}{y}\phi|)\! +\! \lambda \, |\discHess{x}{y}\phi|^2 \opd V_g(y)\opd V_g(x).
\end{equation}
\end{enumerate} 
In words, one could say that in both sampling strategies \ref{S1} and \ref{S3}, first $x$ and $y$ are sampled uniformly in $M$. Then in $\ref{S1}$, $x$ is kept while $y$ is resampled until $y \in D_\eps(x)$. In contrast, in \ref{S3}, both $x$ and $y$ are rejected and resampled if $(x,y) \notin \discprod_\eps$. These strategies are in general not the same (for example if $M$ has a boundary).
All these functionals can be observed as an instance of
\begin{equation}\label{eq:abstract_functional}
   \energy^\eps[\phi] \coloneqq \int_M\!\! \int_{D_\eps^M\!(x)}  \left(\distortion(|\discGrad{x}{y}\phi|) + \lambda \, |\discHess{x}{y}\phi|^2\right) \densitymm(x,y) \opd V_g(y)\opd V_g(x),
\end{equation}
where the density $\densitymm  \in L^1(M\times M)$ satisfies the following conditions:
\begin{enumerate}[label=(D\arabic*), itemsep=.3\baselineskip]
	\item $\int_M \int_{D_\eps^M\!(x)}\densitymm(x,y)\opd V_g(y)\opd V_g(x) = 1$ for all $\eps > 0$. \label{D1}
	\item There exist $c_\rho,C_\rho>0$ with $C_\rho \geq \eps^m\densitymm(x,y) \geq c_\rho$ for almost every $x\in M$, $y\in D^M_\eps(x)$, and $\eps>0$. \label{D2}
	\item $\eps^m \densitymm(x,\parametrization_x(\eps w))$ converges pointwise to some function $\rho(x,w)$ as $\eps\to0$ (using the regularity of $M$) for almost every $x \in M, w \in B_1^m(0)$. \label{D3}
\end{enumerate}
In fact, for all three functionals \eqref{eq:MonteCarlo1}, \eqref{eq:MonteCarlo2}, \eqref{eq:MonteCarlo3}
the limit in \ref{D3} is the same, the constant function $\rho(x,w)=\frac{1}{V_g(M) \omega_m}$ with $\omega_m = \mathcal{L}(B_1^m(0))$.
\begin{remark}\label{rmrk:annulus}
One may also consider more general sampling strategies, an example being sampling on an annulus  {$D_\eps^M(x)\setminus D_{c\eps}^M(x)$} for some constant $0 < c < 1$. In this case, the limit density will be $$\rho(x,w)=\frac{1}{V_g(M)\omega_m(1-c^m)}\chi_{B_1^m(0)\setminus B_c^m(0)(w)}.$$ Note that the lower bound in \ref{D2} is not satisfied in this case. However, all our results still hold, with some minor modifications in the proofs, if the lower bound assumption  only holds for $y$ from an open subset of $D^M_\eps(x)$ of fixed measure, for all $x \in M$.
\end{remark}

Functionals of the type \eqref{eq:abstract_functional} were used in \cite{BoBrMi01} in the context of (fractional) Sobolev spaces with applications in the study of variational problems, e.g., in \cite{AuKo09,LePaScSp15}.

Now, let us rewrite functional \eqref{eq:abstract_functional}
using normal coordinates around each $x$, which will later simplify the limit analysis of the functional. To this end, we note that, for any measurable function $f_x: M \to \R$, we have
\begin{align*}
\int_{D_\eps^M(x)} f_x(y)\opd V_g(y)
= \int_{B_1^{m}(0)}  
f_x(\parametrization_x(\eps w))w^\eps(x,w) \sqrt{\det G_x(\parametrization_x (\eps w))}\eps^m\opd w \label{eq:trafo_final_step}.
\end{align*}
In the above, $w^\eps(x,w)\coloneqq\chi_{U_x^\eps}(\eps w)$,  for $\chi_B$ the characteristic function of a set $B$, 
and $G_x\in\R^{m\times m}$ denotes the matrix representation of the Riemannian metric $g$, in Riemannian normal coordinates around $x$.
Taking $f_x(y)=(\distortion(|\partial_{(x,y)}\phi|)+\lambda |\partial^2_{(x,y)} \phi|^2)\densitymm(x,y)$, one obtains
\begin{equation} \label{eq:integral_over_sphere1}
\energy^\eps[\phi]
=\int_M \int_{B_1^{m}(0)} \left(\distortion(|\partial_{(x,\parametrization_x(\eps w))}\phi|) + \lambda |\partial^2_{(x,\parametrization_x(\eps w))} \phi|^2\right)
\rho^\eps(x,w)  \opd w \opd V_g(x)
\end{equation}
for $\rho^\eps(x,w) \coloneqq w^{\eps}(x,w) 
\sqrt{\det G_x(\parametrization_x(\eps w))} \eps^m \densitymm(x,\parametrization_x(\eps w))$. 
From the compactness of $\overline{M}$ and \ref{D2} we deduce that there exist $\tilde c,\tilde C>0$ with
\begin{equation}\label{eqn:weightBound}
\tilde C_\rho \geq \rho^\eps(x,w) \geq\tilde c_\rho, 
\qquad\text{for all }x\in M,w\in U_x, 0<\eps<r_0.
\end{equation}
Furthermore, using Taylor expansion of $G_{x}(\parametrization_x (\eps w))$ around $\eps=0$ and the local Lipschitz property of the determinant function,  we have~\cite[Proposition 3.1, Lemma 3.5]{Sa96}
\begin{equation}\label{eq:volume_element_expansion}
\det G_x(\parametrization_x(\eps w))=1-\frac13 \text{Ric}(\iota_x\tfrac w{|w|},\iota_x\tfrac w{|w|}) |w|^2 \eps^2 + O(|w|^3\eps^3),
\end{equation}
where $\text{Ric}$ is the Ricci curvature. Due to smoothness of the metric and 
compactness of the manifold the remainder term is bounded uniformly in $x$.
Together with assumption \ref{D3},  and the fact that $w^{\eps}(x,w)$ converges pointwise to $1$ for every $x\in M$
 and $w\in B_1^{m}(0)$,  this implies
\begin{equation}\label{eq:bound-h}
\rho^\eps(x,w) \to \rho(x,w) \text{ as } \eps \to 0,~ \text{ for a.e. } x \in M, w \in B_1^{m}(0).
\end{equation}

%%%%%%%%%%%%%%%%%%%%%%%%%%%%%%%%%%%%%%%%%%%%%%%%%%%%%%%%%%%%%%%%%%

\section{Analysis of the continuous sampling loss functional}\label{sec:analysis}
In this section we study the continuous sampling loss functional \eqref{eq:abstract_functional}, or equivalently \eqref{eq:integral_over_sphere1}. We show existence of a minimizer in an appropriate class of functions.
Furthermore, we show convergence to a local energy functional for vanishing sampling radius $\eps$.
(Note that the continuous sampling loss functional \eqref{eq:abstract_functional} is nonlocal,
due to the double integral over pairs of points.)

\subsection{Preliminaries on function spaces}\label{subsec:function_spaces}
To study the nonlocal energy \eqref{eq:abstract_functional}, we will need suitable Riemannian  
notions of derivatives for functions \mbox{$\phi\colon M \to \R^l$}.
The \emph{Riemannian gradient (Jacobian)} $\Grad\phi(x)=(\Grad\phi_1(x),\ldots,\Grad\phi_l(x))\in(T_xM)^l$ of $\phi$ can be defined via the identity
\begin{equation*}
\frac{\mathrm{d}}{\mathrm{d}t}(\phi_j\circ\exp_x)(tv)\vert_{t=0}=g(\Grad\phi_j(x),v),~ \text{ for all }v\in T_xM.
\end{equation*}
The Riemannian Hessian $\Hess\phi(x)=(\Hess\phi_1(x),\ldots,\Hess\phi_l(x))$ is the linear operator 
$T_xM\to(T_xM)^l$ defined by $\Hess\phi_j(x)[v]=\nabla_v\Grad\phi_j(x)$, $j=1,\ldots,l$,
where $\nabla$ is the Levi-Civita connection and $\nabla_v$ the covariant derivative in direction $v$. We have the identity \cite[Chp.\,5]{Absil09}
\begin{equation*}
\frac{\mathrm{d}^2}{\mathrm{d}^2t}(\phi_j\circ\exp_x(tv))\vert_{t=0}
= g_x(\Hess\phi_j(x)[v],v)~ \text{ for all }v\in T_xM.
\end{equation*}
For notational simplicity we will sometimes apply the Riemannian metric $g$ on $(T_xM)^l\times T_xM$, which then is to be understood componentwise, \ie, we use the notation
$g(A,v) = (g(A_j,v))_{j=1,\ldots, l}$ for an $l$-tuple \mbox{$A=(A_1,\ldots,A_l)$} of tangent vectors.
Similarly, we write $g(A,B)=\sum_{j=1}^lg(A_j,B_j)$ for $A,B\in(T_xM)^l$ and $g(H,K)=\sum_{j=1}^mg(H[v_j],K[v_j])$ for linear operators $H,K:T_xM\to T_xM$, where $v_1,\ldots,v_m$ is any orthonormal basis of $T_xM$.
In the above notation the previous identity reads 
$$\frac{\mathrm{d}^2}{\mathrm{d}^2t}(\phi\circ\exp_x(tv))\vert_{t=0}= g(\Hess\phi(x)[v],v) ~ \text{ for all }v\in T_xM.$$%
Given $(x,y) \in \discprod_{r_0}$, denote by 
\begin{equation*}
  \PT_{y \leftarrow x}\colon T_x M \to T_y M
\end{equation*}
the parallel transport operator along the unique minimizing geodesic $\gamma\colon [0,1]\to M$ with $\gamma(0)=x$ and $\gamma(1)=y$. 
The gradient of a differentiable function is Lipschitz of constant $L_{\Grad}(\phi)$ if
\begin{equation*}
\forall (x,y) \in \discprod_{r_0}, \quad \norm{\PT_{x\leftarrow y}\Grad\phi(y)-\Grad\phi(x)} 
\leq L_{\Grad}(\phi) d_M(x,y),
\end{equation*}
The Hessian of a twice differentiable function $\phi$ is Lipschitz of constant $L_{\Hess}(\phi)$ if
\begin{equation*}
\forall (x,y) \in \discprod_{r_0}, \quad \norm{\PT_{x\leftarrow y}\,\Hess\phi(y)\,\PT_{y\leftarrow x}-\Hess\phi(x)} \leq L_{\Hess}(\phi)d_M(x,y).
\end{equation*}
In both cases, $\|A\|=\sqrt{g(A,A)}$, with the summation convention as defined above. 
We write $\phi \in C^{2,1}(M, \R^l)$ for such functions~\cite[Sec. 10.4]{Bo22}. We have the following estimate, where from now on we use notation
\begin{equation*}
\bar w\coloneqq\tfrac{w}{|w|},\qquad\text{for any }w\in\R^m.
\end{equation*}

\begin{lemma}\label{lemma:taylor} Let $\phi \in C^{2, 1}(M, \R^l)$. Then, for all $x \in M, w\in U_x$ we have 
\begin{align}
\label{eq:first-order-taylor-r}
\abs*{\discGrad{x}{\parametrization_x (w)}\phi - g_x(\Grad\phi(x), \iota_x \bar w)} \leq  &\tfrac 1 2 L_{\Grad}(\phi) |w|, \\
\label{eq:second-order-taylor-r}
\abs*{\discHess{x}{\parametrization_x (w)}\phi - g_x(\Hess\phi(x)[\iota_x \bar w], \iota_x \bar w)} \leq & \tfrac 5 6 L_{\Hess}(\phi) |w|.
\end{align}
\end{lemma}
%%%%%
\begin{proof}
Given $v \in \iota_x(U_x)$ with $\|v\|=1$, the first and second order Taylor expansion of $r\mapsto\phi\circ\exp_x(rv)$ around zero yields
\begin{align*}
\phi(\exp_x(rv))
&=\phi(x)+g_x(\Grad\phi(x), rv)+R_1(x,rv)\\
&=\phi(x)+g_x(\Grad\phi(x), rv)+\tfrac 1 2 g_x(\Hess\phi(x)[rv], r v)+R_2(x,rv),
\end{align*}
where, following \cite[Sec. 10.4]{Bo22}, the remainder terms can be bounded by
$\abs*{R_1(x, r v)} \leq \tfrac 1 2L_{\Grad}(\phi) r^2$ and $\abs*{R_2(x, rv)} \leq \tfrac 1 6L_{\Hess}(\phi) r^3$, respectively.
Choosing $v=\iota_x \bar{w}$ and $r=|w|$ and $r=\frac{|w|}{2}$, respectively, we get identities 
\begin{align*}
\discGrad{x}{\parametrization_x (w)}\phi - g_x(\Grad\phi(x), \iota_x \bar{w})
&=\tfrac{R_1(x,\iota_x \bar{w})}{|w|},\\
\discHess{x}{\parametrization_x (w)}\phi - g_x(\Hess\phi(x)[\iota_x \bar{w}], \iota_x \bar{w})
&=\tfrac{4R_2(x,\iota_x w)-8R_2(x,\iota_x \tfrac{w}{2})}{|w|^2}.
\end{align*}
From these, the claim directly follows via the already verified bounds for the remainder terms $R_1$, $R_2$ on the right hand sides.
\end{proof}
The Sobolev spaces $H^1(M,\R^l)$ and $H^2(M,\R^l)$ are defined as the closure of $C^{2,1}(M,\R^l)$ under the norms
\begin{align*}
\|\phi\|_{H^1(M,\R^l)}^2
&\coloneqq\sum_{j=1}^l\int_M|\phi_j|^2
+g(\Grad\phi_j,\Grad\phi_j)
\opd V_g,\\
\label{eq:sobolev_norm}
\|\phi\|_{H^2(M,\R^l)}^2
&\coloneqq\sum_{j=1}^l\int_M|\phi_j|^2
+g(\Grad\phi_j,\Grad\phi_j)
+g(\Hess\phi_j,\Hess\phi_j)\opd V_g.
\end{align*} 
A classical property of Sobolev spaces on a compact manifold is Rellich's embedding theorem, which implies that $H^2(M,\R^l)$ is compactly embedded in $H^1(M,\R^l)$ and in $L^2(M,\R^l)$.
For further details on Sobolev spaces on 
manifolds we refer to \cite{Hebey96} and to \cite{aubin_espaces_1976} for the Rellich embedding theorem in particular.
The subset 
of functions with zero mean is denoted by $\dot{H}^2(M,\R^l)$. 
We will also make use of the following equivalent norms.
\begin{proposition}\label{prop:norm_equivalence}
Let 
$V=\cone_{1,\kappa}$, where $\cone_{1,\kappa}\subset \R^m$ is a cone as defined in \ref{M2}. Then, the quantity
 \begin{equation}
\absav{W} \coloneqq \left(\int_{V} g(W,\iota_x\bar w)^2 \opd w\right)^{\frac12}
\end{equation}
for $W\in T_x M$ defines a norm on $T_x M$, which is uniformly equivalent to the standard norm based on the metric.
Let $L_{\text{sym}}(T_x M, T_xM)$ denote the space of symmetric endomorphisms on $T_x M$. Then a norm on this space is given by
\begin{equation}
\absav{A} \coloneqq \left(\int_{V} g(A[\iota_x\bar w],\iota_x\bar w)^2 \opd w\right)^{\frac12},
\end{equation}
and it is uniformly equivalent to the standard norm on $L_{\text{sym}}(T_x M, T_xM)$.
As a consequence, the norm
\begin{equation}
	\norm{\phi}_{H^2_{\text{av}}(M,\R^l)} \coloneqq \left(\int_M \absav{\Grad\phi(x)}^2  +  \absav{\Hess\phi(x)}^2 \opd V_g(x)\right)^{\frac 1 2}
\end{equation}
is equivalent to the standard $H^2$-norm on $\dot H^2(M,\R^l)$.
\end{proposition}
\begin{proof}
It is obvious that $W \mapsto \absav{W}$ is 1-homogeneous and satisfies the triangle inequality; its positive definiteness will follow from the norm equivalence shown below.
Using the notation $\|W\|^2={g}(W,W)$, we get 
$\absav{W}^2 \leq \mathcal{L}(V) \norm{W}^2$
by the Cauchy--Schwarz inequality.
Furthermore, using the isometry of $\iota_x$, we observe
\begin{equation*}
 	\absav{W}^2 
			=\int_V (\iota_x^{-1}(W) \cdot \bar w)^2 \opd w
			\geq c |\iota_x^{-1}(W)|^2
 			=  c \norm{W}^2,
 \end{equation*}
where $c>0$ is a constant independent of $x$.
Namely, the second term above is a norm on $\R^m$. This is due to the fact that $v \cdot w=0$ for all $w \in V$ implies $v=0$, since $V$ contains an open set. One follows an analogous argument for the norm on $L_{\text{sym}}(T_x M, T_x M)$. Note that for any linear operator $B\colon \R^m \to \R^m$, $Bw\cdot  w = 0$ for every $w \in V$ implies that $B$ is skew symmetric. Applying this observation to ${B=\iota_x \circ A \circ \iota_x} \in L_{\text{sym}}(\R^m, \R^m)$, we obtain ${\iota_x \circ A \circ \iota_x = 0}$.
From this it immediately follows that the norm $\norm{\cdot}_{H^2_{\text{av}}(M,\R^l)}^2$ is equivalent to the standard $H^2(M,\R^l)$-seminorm $\sum_{j=1}^l\int_Mg(\Grad\phi_j,\Grad\phi_j)+g(\Hess\phi_j,\Hess\phi_j)\opd V_g$, which is equivalent to the $H^2(M,\R^l)$-norm defined above by Poincar\'e's inequality on $\dot H^2(M,\R^l)$.
\end{proof}
%%%%%%%%%%%%%%%%%%%%%%%%%%%%%%%%%%%%%%%%%%%%%%%%%%%%%%%%%%

\subsection{Existence of minimizer to nonlocal energy}\label{subsec:discrete_existence}
In this section we study the existence of a minimizer for the nonlocal energy \eqref{eq:abstract_functional}.
One way to approach this problem would be to study it on the function space naturally associated with the energy $\energy^\eps$. This space would be the completion of $C^{2,1}(M,\R^l)$ under the Hilbert space norm
$\|\phi\|_\eps^2\coloneqq\|\phi\|_{L^2(M,\R^l)}^2+|\phi|_\eps^2$
with
\begin{equation*}
|\phi|_\eps^2\coloneqq\int_M\int_{D_\eps^M(x)}(|\discGrad{x}{y}\phi|^2 + |\discHess{x}{y}\phi|^2)\densitymm(x,y)\opd V_g(y)\opd V_g(x).
\end{equation*}
On this space the energy would be coercive and likely admit a minimizer;
indeed (at least for our choice of $\distortion$), there exists $C>0$ such that $-C+\frac1C|\phi|_\eps^2\leq\energy^\eps[\phi]\leq C+C|\phi|_\eps^2$, independent of $\eps$.
This function space is closely related to $H^2(M,\R^l)$:
in \cite{BoBrMi01,Bo07} it is shown that, if $M=\R^m$ and $\phi\in L^2(M,\R^l)$ satisfies \mbox{$\liminf_{\eps\to0}|\phi|_\eps<\infty$}, then $\phi \in H^2(M, \R^l)$. A corresponding result just for the first order term was also shown for smooth, compact and connected Riemannian manifolds in \cite{KrMo18}. There are also equicoercivity or compactness results for sequences ${\phi^\eps}$ with uniformly bounded $\{|\phi^\eps|_\eps\}_{\eps>0}$ (see e.g.\ \cite[Corollary 6]{BoBrMi01}). All this would ultimately allow us to prove that minimizers of $\energy^\eps$ converge to a minimizer of a limit functional. 

We will not follow this approach, and instead only look for minimizers on a more restricted $\eps$-dependent set $\NNSpace$ of functions from $M$ to $\R^l$, which will be realized later via neural networks.
We prefer this latter approach for two reasons: First, we would need to translate the theory of \cite{BoBrMi01,Bo07} to functions on manifolds.
Second, the compactness results are sensitive to the particular choice of the sampling weight $\densitymm$.
In particular, while equicoercivity will probably hold for our examples \eqref{eq:MonteCarlo1}-\eqref{eq:MonteCarlo3},
one might also think of sampling weights $\densitymm(x,\cdot)$ that are supported on an annulus as noted in Remark~\ref{rmrk:annulus}.
in which case equicoercivity might potentially be lost (\cf~\cite[Counterexample 1]{BoBrMi01}), while it still holds in the set $\NNSpace$.

Since the energy $\energy^\eps$ is translation invariant,
functions in $\NNSpace$ are (without loss of generality) assumed to have zero mean.
In addition, we impose the following conditions on $\NNSpace$:
\begin{enumerate}[label=(H\arabic*)]
\item(closedness in $\dot L^2$) \label{H1} For every $0<\eps<r_0$ ,
we have that $\NNSpace $ is closed as a subset of 
$\dot{L}^2(M,\R^l)$, the set of $L^2$ functions on $M$ with zero mean.
\item(boundedness) \label{H2} $\NNSpace \subset C^{2,1}(M,\R^l)$,
and there exists a constant $C_L\geq 0$ such that
\begin{equation*}
\limsup_{\eps\to 0}\sup_{\phi^\eps \in \NNSpace}\eps \left(L_{\Grad}(\phi^\eps) + L_{\Hess}(\phi^\eps) \right)\leq C_L ,
\end{equation*}
\item(density in $\dot{H}^2(M,\R^l)$) \label{H3} 
 For every $\phi \in \dot{H}^2(M,\R^l)$, there exists a sequence $\{\phi^\eps\}_{\eps>0}$ with $\phi^\eps \in \NNSpace$ such that $\lim\limits_{\eps \to 0}\|\phi^\eps - \phi\|_{H^2(M,\R^l)}=0$.
\end{enumerate}
Here the notation $\eps \to 0$ is short for a sequence $(\eps_k)_{k\in \N}$ with $\eps_k\to 0$ for $k\to \infty$.
These conditions can in particular be satisfied if the functions in $\NNSpace$ are realizations of deep neural networks with sufficiently smooth nonlinear activation functions (such as the \emph{sigmoid/logistic} function or the \emph{softplus} function) and $\eps$-dependent bounds on the network parameters and weights. 
Indeed, bounds on the growth of derivatives of these realizations can be computed explicitly in terms of the numbers of layers and a priori bounds on the network weights by applying the chain rule (\cf~\cref{subsec:autoencoder}).
The closedness of the spaces $\NNSpace$ was shown in \cite[Proposition 3.5]{PeRaVo21}, while the density property was shown in an already classical article by Hornik et al.~\cite{HoStWh90}. 
Note that all these results were actually shown for neural networks defined on Euclidean domains, but they can be transferred to smooth $m$-dimensional manifolds $(M,g)$ embedded in $\R^n$ such that the metric $g$ is equivalent to the metric induced by the embedding. 
For more details we refer to \cite[Section 7]{BoGrKuPe19}. In \cref{sec:experiments}, we will provide more information on the specific network architecture employed by us to satisfy \ref{H1}--\ref{H3}.

The following theorem provides the existence of a minimizer for the $\NNSpace$-restricted energy
\begin{equation}\label{eqn:restrictedNonlocalFunctional}
\energy_{\mathcal F}^\eps[\phi]\coloneqq\begin{cases}
\energy^\eps[\phi]
&\text{if }\phi\in \NNSpace,\\
\infty&\text{else,}
\end{cases}
\end{equation}
and in addition the uniform boundedness of all minimizers in $H^2(M,\R^l)$, independent of $\epsilon$.

\begin{theorem}\label{thm:strong_discrete_existence}
Let conditions \ref{H1}--\ref{H2} be satisfied.
Then, for every small enough $\eps>0$, 
there exists a minimizer $\phi^\eps$ of energy \eqref{eqn:restrictedNonlocalFunctional}, and $\|\phi^\eps\|_{H^2(M,\R^l)} \leq C$ 
for a constant $C>0$ independent of $\eps$.
\end{theorem}

\begin{proof}
Below, we will exploit the energy reformulation \eqref{eq:integral_over_sphere1} of $\energy^\eps$.
We have  \mbox{$\parametrization_x(\cone_{r_0,\kappa})\subset D^M_{r_0}(x)$} by condition \ref{M2} on $M$ and thus in particular ${\parametrization_x(\cone_{\eps,\kappa})\subset D_\eps^M(x)}$ for all $\eps < r_0$. 
Let $V=\cone_{1, \kappa}$ as in \Cref{prop:norm_equivalence}. 
Recalling the definition of $w^\eps(x, w)=\chi_{U_x^\eps}(\eps w)$, we see 
\begin{equation}\label{eq:w-constant}
w^\eps(x, w)=1 \text{ for all }  x \in M, w \in V.
\end{equation}
Using first \eqref{eqn:weightBound} (which comes from \ref{D2}) and then the inequality $\abs a^2 \geq \frac{1}{2}\abs{b}^2-\abs{a-b}^2$ for $a=\partial_{(x, \parametrization_x(\eps w))}\phi$ and $b=g_x(\Grad\phi(x), \iota_x w)$, we get 
\begin{align*}
& \int_M \int_{D^M_\eps(x)} \abs{\partial_{(x,y)}\phi}^2 \densitymm(x,y) \opd V_g(y)\opd V_g(x)\\
	&= \int_M \int_{B_1^{m}(0)}
	\abs{\partial_{(x,\parametrization_x (\eps w))}\phi}^2 \rho^\eps(x,w)\opd w\opd V_g(x)\\
	&\geq
	\frac{\tilde c_\rho}{2} \int_M \int_{V} \abs{g_x(\Grad\phi(x), \iota_x\bar w)}^2 \opd w\opd V_g(x) \\  
	& \quad - \tilde c_\rho \int_M \int_{V} 
	\abs{g_x(\Grad\phi(x), \iota_x\bar w)-\partial_{(x, \parametrization_x(\eps w))}\phi}^2\opd w\opd V_g(x) 	
	 \\
	 & \geq \frac{\tilde c_\rho}2 \int_M \absav{\Grad\phi}^2 \opd V_g - C \eps^2L_{\Grad}(\phi)^2
\end{align*}
for a constant $C>0$,
where in the last step we used \Cref{lemma:taylor}.
We analogously obtain 
\begin{equation*}
\!\int_M\! \int_{D^M_\eps(x)}\! \abs{\partial^2_{(x,y)}\phi}^2 \densitymm(x,y)   \opd V_g(y)\opd V_g(x)
	 \!\geq\!  \frac{\tilde{c}_\rho}{2} \int_M \absav{ \Hess\phi}^2 \opd V_g - C \eps^2 L_{\Hess}(\phi)^2.
\end{equation*}
Summing both terms and applying \Cref{prop:norm_equivalence}, one sees that there exists 
$C>0$ such that
\begin{equation}
\|\phi\|^2_{H^2(M,\R^l)} \leq C\left(\energy^\eps[\phi] + C+ \eps^2 \left(L_{\Grad}(\phi) + L_{\Hess}(\phi)\right)^2\right)
 \leq C\left(\energy^\eps[\phi] + C+ C_L^2\right),
\end{equation}
for every $\phi \in \NNSpace$ and $\eps$ small enough (\cf~assumption \ref{H2}).
Obviously, $\energy^\eps[0]$ is a finite upper bound for the energy of the minimizer. Thus,
for any fixed $\eps>0$, there exists a minimizing sequence $\{\phi^\eps_j\}_{j \in \N}$ with $\|\phi^\eps_j\|_{H^2(M,\R^l)} \leq C(\energy^\eps[0]+C_L^2)$ and $\lim_{j \to \infty} \energy^\eps[\phi^\eps_j]=\inf_{\phi \in \NNSpace} \energy^\eps[\phi]$. 
Hence, there exists a weakly convergent subsequence (after relabeling) $\phi^\eps_j \rightharpoonup \phi^\eps$ in $H^2(M,\R^l)$. By weak lower semicontinuity of the norm, we have $\|\phi^{\eps}\|_{H^2(M,\R^l)} \leq C(\energy^\eps[0]+C_L^2)$. Furthermore,
 we have $\phi^\eps_j \to \phi^\eps$ in $\dot L^2(M,\R^l)$ by the compact embedding property, so that \ref{H1} implies $\phi^\eps \in \NNSpace$.
Extracting a further subsequence, we even obtain pointwise almost everywhere convergence, so that
by Fatou's lemma we have $\liminf_{j \to \infty} \energy^\eps[\phi^\eps_j] \geq \energy^\eps[\phi^\eps]=\energy_{\mathcal F}^\eps[\phi^\eps]$, which finishes the proof.
\end{proof}
Uniqueness of a minimizer to \eqref{eqn:restrictedNonlocalFunctional} cannot be expected:
Indeed, the energy $\energy^\eps$ is invariant under composition of its argument from the left with a rigid motion or a reflection. 
Even apart from this invariance, the nonconvexity of the first integrand in \eqref{eq:abstract_functional}, {which is unavoidable when promoting isometries,} may prevent uniqueness.
However, whenever $M$ is intrinsically flat and homeomorphic to the $m$-disc, there is a unique minimizer of $\energy^\eps$ up to rigid motion and reflection
~\cite[Proposition 1]{BrRaRu21}. Being able to find flat embeddings may actually be quite relevant in applications,
as it was noticed in \cite{ShKuFl18} that generative image manifolds have almost no curvature.

%%%%%%%%%%%%%%%%%%%%%%%%%%%%%%%%%%%%%%%%%%%%%%%%%%%%%%%%%%%%%%%%%%%%%%
\subsection{The limit for vanishing sampling radius}
\label{subsec:limit}
In this section we study convergence of the functional \eqref{eqn:restrictedNonlocalFunctional} as the sampling radius $\eps$ tends to $0$. 
The limit loss functional is local, promotes low distortion and low bending embeddings, and reads
\begin{align}
\energy[\phi] &\coloneqq \int_M \Distortion(\Grad\phi(x))
+\lambda \|\Hess\phi(x)\|^2\opd V_g(x), \label{eq:limitE}
\qquad\text{with }\\
\label{eq:Gamma_choice}
\Distortion(W)&\coloneqq \int_{B_1^{m}(0)} \distortion(|g_x(W,\iota_x\bar w)|)  \rho(x,w) \opd w\,, \\
\label{eq:HessNorm_choice}
\Vert A\Vert^2 &\coloneqq
\int_{B_1^{m}(0)} |g_x(A[\iota_x\bar w],\iota_x\bar w)|^2 \rho(x,w) \opd w  \,,
\end{align}
for all $W\in T_xM$ and $A\in L(T_xM,T_xM)$ and $\rho$ from \ref{D3},
as stated in our main result.
\begin{theorem}[Mosco-convergence]\label{thm:Mosco_smooth} Let \ref{H2}--\ref{H3} be satisfied with $C_L=0$.
Then the nonlocal regularization energies $\{\energy_{\mathcal F}^\eps\}_{\eps >0}$ given by \eqref{eqn:restrictedNonlocalFunctional} converge to the continuous regularization energy $\energy$ given by \eqref{eq:limitE} as $\epsilon \to 0$, 
in the sense of Mosco in $\dot H^2(M,\R^l)$.
\end{theorem}
\begin{remark}
An isometric embedding $\phi:M\to\R^l$ is characterized
by $\Grad\phi(x)$ being orthogonal in any point $x\in M$, or equivalently 
$\vert g_x(\Grad\phi(x),\iota_xv)\vert =1$ for all $v\in S^{m-1}$, where $S^{m-1} \coloneqq\partial B_1^m(0)$.
Therefore, the first term of \eqref{eq:limitE} penalizes deviation from an isometric embedding, as it is zero for locally isometric $\phi:M\to \R^l$ and strictly positive otherwise.
Similarly, extrinsic bending of the embedding $\phi:M\to \R^l$ manifests as a non-vanishing Riemannian Hessian.
Since the term \eqref{eq:HessNorm_choice} defines a squared norm on the space
$L_{\text{sym}}(T_xM,T_xM)^l$ by \Cref{prop:norm_equivalence},
the second term in \eqref{eq:limitE} penalizes extrinsic bending.
\end{remark}
\begin{proof}[Proof of \Cref{thm:Mosco_smooth}]
We have to verify the two properties \cite{Mo69a} defining Mosco-convergence: \\
\setlength{\mylongest}{\widthof{$\limsup$}}
\addtolength{\mylongest}{\labelsep}

\begin{minipage}{0.9\textwidth}
\begin{enumerate}[align=left,labelwidth=\mylongest,leftmargin=!]
\item[$\liminf$:] For every sequence $\{\phi^\epsilon\}_{\eps>0}$ converging weakly to 
$\phi$ in $\dot H^2(M,\R^l)$ as $\eps\to 0$, one has that \mbox{$\liminf\limits_{\epsilon \to 0} \energy_{\mathcal F}^\eps[\phi^\epsilon] \geq \energy[\phi]$}.
\item[$\limsup$:] For every $\phi \in \dot{H}^2(M,\R^l)$, there exists a sequence $\{\phi^\epsilon\}_{\eps>0}$ converging for $\eps\to 0$ strongly to $\phi$ in $\dot H^2(M,\R^l)$  with
$\limsup\limits_{\epsilon \to 0} \energy_{\mathcal F}^\eps[\phi^\epsilon] \leq \energy[\phi]$.
\end{enumerate}  
\end{minipage}\\[0.5em]
For ease of notation,  we will sometimes use an expression of the form $f(x,w)$ to actually indicate the mapping {$M\times B_1^m(0)\ni(x,w)\mapsto f(x,w)$.}
It will always be clear from the context when this is intended.
{For easier reference, we also recall the form of our nonlocal functional
\begin{equation} \label{eq:integral_over_sphere_repeat}
\energy^\eps[\phi]
=\int_M \int_{B_1^{m}(0)} \left(\distortion(|\partial_{(x,\parametrization_x(\eps w))}\phi|) + \lambda |\partial^2_{(x,\parametrization_x(\eps w))} \phi|^2\right)
\rho^\eps(x,w)  \opd w \opd V_g(x).
\end{equation}
}
\medskip
\subsubsection*{Proof of the $\liminf$ property.}
Let $\phi^\eps \rightharpoonup \phi$ in $\dot H^2(M,\R^l)$. Putting aside the trivial case of $\liminf\limits_{\eps \to 0} \energy_{\mathcal F}^\eps[\phi^\eps]=+\infty$, we may assume that $\phi^\eps \in \NNSpace$ for every $\eps>0$ along a subsequence.
We first estimate
\begin{align*}
&\limsup_{\eps \to 0}\| \discGrad{x}{\parametrization_x (\eps w)}\phi^\eps - g_x(\Grad\phi(x),\iota_x\bar w) \|_{L^2(M \times B_1^{m}(0))} \\
&\leq 
\limsup_{\eps \to 0}\|\discGrad{x}{\parametrization_x (\eps w)}\phi^\eps - g_x(\Grad\phi^\eps(x),\iota_x\bar w)  \|_{L^2(M \times B_1^{m}(0))} \\
 &\quad + \limsup_{\eps \to 0}\|g_x(\Grad\phi^\eps(x),\iota_x\bar w) -  g_x(\Grad\phi(x),\iota_x\bar w)\|_{L^2(M \times B_1^{m}(0))}.
\end{align*}
By \eqref{eq:first-order-taylor-r} the first summand can be estimated as
\begin{equation*}
\|  \discGrad{x}{\parametrization_x (\eps w)}\phi^\eps -g_x(\Grad\phi^\eps(x),\iota_x\bar w)  \|_{L^2(M \times B_1^{m}(0))}
\leq C  \eps L_{\Grad}(\phi^\eps),
\end{equation*}
which by \ref{H2} converges to $0$.

To show the same for the second summand, we observe that $\Grad\phi^\eps \to \Grad\phi$ in $L^2(M,(TM)^{l})$ by the compact embedding of $\dot H^2(M,\R^l)$ into $H^1(M,\R^l)$. 
Thus we obtain $g_x(\Grad\phi^\eps(x),\iota_x \bar{w}) \to g_x(\Grad\phi(x),\iota_x \bar{w})$ in $L^2(M \times B^m_1(0),\R^l)$ since the function 
\begin{equation*}
    V \mapsto [(x, w) \mapsto g_x(V(x), \iota_x \bar w)]
  \end{equation*}
is a bounded linear function from $L^2(M,(TM)^l)$ to $L^2(M \times B^m_1(0),\R^l)$.
Hence, exploiting \eqref{eq:bound-h} and passing to a subsequence without relabeling, for almost every $x \in M, w \in B_1^{m}(0)$ we have
\begin{equation*}
 \distortion(|\discGrad{x}{\parametrization_x (\eps w)}\phi^\eps|) {\rho^\eps(x,w)} \to  \distortion(|g_x(\Grad\phi(x),\iota_x\bar w)|) \rho(x,w).
\end{equation*}
Applying Fatou's lemma we obtain
\begin{align}\label{eq:liminf_first_part}
&\liminf_{\eps \to 0}\int_M \int_{B_1^{m}(0)} 
\distortion(|\discGrad{x}{\parametrization_x (\eps w)}\phi^\eps|)  \rho^\eps(x,w) \opd w\opd V_g(x) \\
&\geq\int_M \int_{B_1^{m}(0)} \distortion(|g_x(\Grad \phi(x),\iota_x\bar w)|) \rho(x,w) \opd w\opd V_g(x).\nonumber
\end{align}
To handle the second order term in \eqref{eq:integral_over_sphere_repeat}, we write
\begin{align*}
&\liminf_{\eps \to 0}\|\discHess{x}{\parametrization_x (\eps w)}\phi^\eps \sqrt{\rho^\eps(x,w)}\|_{L^2(M \times B_1^{m}(0))}\\
& \geq  \liminf_{\eps \to 0} \| g_x(\Hess\phi^\eps(x)[\iota_x\bar w],\iota_x\bar w) \sqrt{\rho^\eps(x,w)}\|_{L^2(M \times B_1^{m}(0))} \\
& \quad  - \limsup_{\eps \to 0} \left\|\left(\discHess{x}{\parametrization_x (\eps w)}\phi^\eps - g_x(\Hess\phi^\eps(x)[\iota_x\bar w],\iota_x\bar w)\right)\sqrt{\rho^\eps(x,w)}\right\|_{L^2(M \times B_1^{m}(0))}.
\end{align*}
By \eqref{eq:second-order-taylor-r} and \eqref{eqn:weightBound} the last term can be estimated as
\begin{equation*}\label{eq:R2_vanish}
\left\|\left(\discHess{x}{\parametrization_x (\eps w)}\phi^\eps \!-\! g_x(\Hess\phi^\eps(x)[\iota_x\bar w],\iota_x\bar w)\right) \sqrt{\rho^\eps(x,w)}\right\|_{L^2(M \times B_1^m(0))} \leq C\eps L_{\Hess}(\phi^\eps),
\end{equation*}
which vanishes in the limit by \ref{H2}. 
To bound the first term, we first observe that 
$$g_x(\Hess\phi^\eps(x)[\iota_x \bar{w}],\iota_x \bar{w}) \rightharpoonup g_x(\Hess\phi(x)[\iota_x \bar{w}],\iota_x\bar{w})~~~ \text{in } L^2(M \times B^m_1(0),\R^l),$$
since the function 
\begin{equation*}
    A \mapsto [(x, w) \mapsto g_x(A(x)[\iota_x \bar w], \iota_x \bar w)]
\end{equation*}
is a bounded linear function from $L^2(M,{L}_{\text{sym}}(TM,TM)^l)$ to $L^2(M \times B^m_1(0), \R^l)$ and $\Hess\phi^\eps$ converges weakly to $\Hess\phi$ in the former space.
 
Furthermore, by the uniform boundedness \eqref{eqn:weightBound} and pointwise convergence \eqref{eq:bound-h} of $\rho^\eps$ {, using the dominated convergence theorem,} we get 
\begin{equation}\label{eq:second_order_weak}
 g_x(\Hess\phi^\eps(x)[\iota_x\bar w],\iota_x\bar w) \sqrt{\rho^\eps(x,w)} \rightharpoonup  g_x(\Hess\phi(x)[\iota_x\bar w],\iota_x\bar w) \sqrt{\rho(x,w)}
\end{equation}
in $L^2(M \times B^m_1(0) ,\R^l).$
Using the weak lower semicontinuity of the norm, this implies
\begin{align*}
&\liminf_{\eps \to 0} \| g_x(\Hess\phi^\eps(x)[\iota_x\bar w],\iota_x\bar w) \sqrt{\rho^\eps(x,w)}\|_{L^2(M \times B_1^{m}(0))} \\
&\geq \| g_x(\Hess\phi(x)[\iota_x\bar w],\iota_x\bar w) \sqrt{\rho(x,w)}\|_{L^2(M \times B_1^m(0))}.
\end{align*}
We thus proved
\begin{align*} 
& \liminf_{\eps \to 0}\int_M \int_{B_1^m(0)}  \lambda |\discHess{x}{\parametrization_x (\eps w)}\phi^\eps|^2 \rho^\eps(x,w) \opd w\opd V_g(x)\\
& \geq  \int_M \int_{B_1^m(0)} \lambda |g_x(\Hess\phi[\iota_x\bar w],\iota_x\bar w)|^2 \rho(x,w) \opd w\opd V_g(x).
\end{align*}
Finally, combining this last inequality with \eqref{eq:liminf_first_part}, we get the desired inequality.
\medskip

\subsubsection*{Proof of the $\limsup$ property.}
By the density assumption \ref{H3}, for every $\phi \in \dot{H}^2(M,\R^l)$ there exists a sequence $\{\phi^\eps\}_{\eps >0}$ such that $\phi^\eps \in \NNSpace$ and $\phi^\eps \to \phi$ in $H^2(M,\R^l)$ as $\eps \to 0$.
Then we can repeat the argument from the $\liminf$ property to prove that
\begin{equation*}
\discGrad{x}{\parametrization_x (\eps w)}\phi^\eps \to g_x(\Grad \phi(x),\iota_x\bar w)~ \text{in}~ L^2(M \times B_1^m(0)).
\end{equation*} 
Now, one uses uniform boundedness  \eqref{eqn:weightBound} and 
pointwise convergence \eqref{eq:bound-h} of $\rho^\eps$ and, 
taking into account the splitting of $\gamma(s)$ into $|s|^2$ and the uniformly bounded $\gamma(s) - |s|^2$,
applies the dominated convergence theorem 
to obtain
\begin{align}\label{eq:limsup_first_part}
&\limsup_{\eps \to 0}\int_M \int_{B_1^m(0)}  \distortion(|\discGrad{x}{\parametrization_x (\eps w)}\phi^\eps|) \rho^\eps(x,w) \opd w\opd V_g(x) \\
&=\int_M \int_{B_1^m(0)}\distortion(|g_x(\Grad\phi(x),\iota_x\bar w)|) \rho(x,w) \opd w\opd V_g(x).\nonumber
\end{align}
For the second order term in \eqref{eq:integral_over_sphere_repeat} we write
\begin{align*}
&\limsup_{\eps \to 0}\|\discHess{x}{\parametrization_x (\eps w)}\phi^\eps \sqrt{\rho^\eps(x,w)}\|_{L^2(M \times B_1^m(0))}\\
& \leq  \limsup_{\eps \to 0} \| g_x(\Hess\phi^\eps(x)[\iota_x\bar w],\iota_x\bar w) \sqrt{\rho^\eps(x,w)}\|_{L^2(M \times B_1^m(0))}\\
  &\quad  + \limsup_{\eps \to 0} \|\left(\discHess{x}{\parametrization_x (\eps w)}\phi^\eps - g_x(\Hess\phi^\eps(x)[\iota_x\bar w],\iota_x\bar w)\right)\sqrt{\rho^\eps(x,w)}\|_{L^2(M \times B_1^m(0))},
\end{align*}
where the last term vanishes as in the proof of the $\liminf$ property. For the first term 
we show
\begin{equation*}
 g_x(\Hess\phi^\eps(x)[\iota_x\bar w],\iota_x\bar w) \sqrt{\rho^\eps(x,w)} \to  g_x(\Hess\phi(x)[\iota_x\bar w],\iota_x\bar w) \sqrt{\rho(x,w)}
\end{equation*}
in $L^2(M \times B^m_1(0),\R^l)$,
with an analogous argument to the one used to obtain \eqref{eq:second_order_weak}.
Thus, again using the dominated convergence theorem, we finally get
	\begin{align*} 
		&\limsup_{\eps \to 0}\int_M \int_{B_1^m(0)}  \lambda |\discHess{x}{\parametrization_x (\eps w)}\phi^\eps|^2 {\rho^\eps(x,w)} \opd w\opd V_g(x) \\
		 &\leq\int_M \int_{B_1^m(0)} \lambda |g_x(\Hess\phi(x)[\iota_x\bar w],\iota_x\bar w)|^2 \rho(x,w) \opd w\opd V_g(x),
	\end{align*}
	which together with \eqref{eq:limsup_first_part} proves the inequality {and finishes the proof of the theorem}.
\end{proof}
The existence of a minimizer for \eqref{eq:limitE} can be established using the direct method of the calculus of variations~\cite[Theorem 2]{BrRaRu21}. 
As for $\energy_{\mathcal F}^\eps$, the minimizer (modulo a rigid motion or reflection) is in general not unique due to the isometry promoting term.
Based on the existence and boundedness result from \Cref{thm:strong_discrete_existence} and the Mosco-convergence we can show convergence of minimizers by a standard procedure in $\Gamma$-convergence theory (see e.g.\ \cite{Br02}).
\begin{corollary}\label{thm:convergence_final}
Let \ref{H1}--\ref{H3} hold with $C_L=0$.
Then for every sequence $\{\phi^\eps\}_{\eps>0}$ of minimizers for the energies $\{\mathcal{E}_{\mathcal F}^\eps\}_{\eps>0}$ there exists a subsequence converging weakly in $H^2(M,\R^l)$ to a minimizer $\phi$ for the energy $\mathcal{E}$. Furthermore, the corresponding function values $\{\energy_{\mathcal F}^\eps[\phi^\eps]\}_{\eps>0}$ converge to $\mathcal{E}[\phi]$.
\end{corollary}

%%%%%%%%%%%%%%%%%%%%%%%%%%%%%%%%%%%%%%%%%%%%%%%%%%%%%%%%

\section{Numerical experiments}
\label{sec:experiments}
As a testbed, we purposely use image manifolds where the manifold is explicitly known and where $\av_M$ and $d_M$ can be explicitly computed. 
Note that for various more general shape manifolds with application in imaging, algorithmic tools to compute suitable approximations of $\av_M$ and $d_M$ are available. Examples of such shape spaces are 
LDDMM \cite{Yo10}, metamorphosis \cite{TrYo05a}, or optimal transport \cite{CuPe19}.
For two reasons, using these tools for training the encoder is beyond the scope of this article: First, these methods are still computationally demanding, such that it would be quite involved to even generate a moderately large set of training data.
To overcome this difficulty one could replace the local distance computation and the local Riemannian average by computationally more efficient approximations. This would require a careful analysis of the approximation error.
Second, while the input data may lie on a low dimensional hidden manifold, the methods will in general not compute intrinsic distances and averages within this manifold. Hence, our analysis would have to be adapted to such perturbed, non exact training data.

%%%%%%%%%%%%%%%%%%%%%%%%%%%%%%%%%%%%%%%
\subsection{Different input manifolds} \label{subsec:setup}
In the same spirit as \cite{DoGr05}, we consider image data that implicitly represent four different manifolds:
\smallskip

\labeltext{({\bf S})}{dataset:s} \textbf{S}undial Shadows ({\bf S}). This dataset has the upper hemisphere  $S^2 \cap \{x_3\geq0\}$ as its underlying hidden manifold, corresponding to possible positions of a lightsource. Inspired by \cite{OrYaHe20}, we generate images by casting a shadow of a vertical rod on a plane from all these positions. These images form our input manifold $M$ with metric induced from $S^2$, so that the distance on $M$ is the geodesic distance on $S^2$, ${d_M}(x,y)=\arccos(x^Ty)$ for $x,y \in S^2$. Contrary to \cite{OrYaHe20}, we do not render these images with a 3D engine, but 
simply approximate the shadows by Gaussians
(\cf~\cref{fig:masterpiece}): 
A point $x\in S^2$ 
is first mapped onto the plane by mapping it onto the point $x_p \in \R^2$ defined as the intersection of the plane with the line through $x$ and the rod tip.
We then use a Gaussian function centered at $x_p/2$ 
with variance $\abs {x_p}^2$ in direction $x_p$ and a fixed small variance in the orthogonal direction.  
 We use a resolution of $64\times 64$ pixels and sampling strategy \ref{S3}.
 
%%%%%%%%%%%%%%%%%%%%%%%%%%%%%
\labeltext{({\bf R})}{dataset:r}\textbf{R}otating Object ({\bf R}). This dataset has the group of rotations $SO(3)$ as hidden manifold. 
We generate images forming the manifold $M$ 
by projecting an arbitrarily rotated three-dimensional object (here a toy cow model\footnote{https://www.cs.cmu.edu/~kmcrane/Projects/ModelRepository/})
to the viewing plane. 
We use \texttt{Pytorch3D} \cite{ravi2020pytorch3d} to render the images for the training.
We use the induced distance from $SO(3)$ equipped with the metric $g_x(v,w)=\frac 1 2 \tr(v^Tw)$ for $x \in SO(3)$, $v,w \in T_xSO(3)$ as distance on $M$. This geodesic distance can be computed for example via quaternions as ${d_M}(x, y)=2\arccos(\abs{ q(x) \cdot q(y)})$ \cite{Hu09}, for $x,y \in SO(3)$ and $q(x),q(y)$ their representation as quaternions. We consider an image resolution of $128\times 128$ pixels and sampling strategy \ref{S2}, which can also be implemented via the quaternion representation.

%%%%%%%%%%%%%%%%%%%%%%%%%%%%%
\labeltext{({\bf A})}{dataset:a}\textbf{A}rc rotating around two specified axes ({\bf A}). This dataset has the Klein bottle $[0,1]^2 /\!\sim$ as the hidden manifold, 
where $\sim$ represents the glueing operation on the boundary of $[0,1]^2$ defining the Klein bottle, \ie\ we identify $(x_1,0)$ and $(x_1,1)$ as well as $(0,x_2)$ and $(1,1-x_2)$. To generate the image representation of $(x_1,x_2)\in[0,1]^2/\sim$ as a point in the manifold $M$, we use Blender \footnote{https://www.blender.org/} to render 
a bent, capped tube representing the arc $[0,1]\ni t \mapsto (\cos(t\pi), \sin(t\pi), 0)$,
rotated by the two angles $\alpha=\pi x_1$ and $\beta = \pi(2x_2 + \frac 1 2)$ (\cf~\cref{fig:arcs-summary}, top row).
The tube is colored using a transition from purple inside to yellow outside. 
We use the flat Euclidean geometry, sampling strategy \ref{S3} on $[0,1]^2/\sim$ and an image resolution of $128\times 128$ pixels. 

%%%%%%%%%%%%%%%%%%%%%%%%%%%%%

\labeltext{({\bf E})}{dataset:e}\textbf{E}llipses (\textbf{E}). We consider rotations and translations of an ellipse of fixed aspect ratio $a>0$ and scale $s>0$ given by the implicit equation $h_a(z/s)\leq 0$ with $h_a(z) = z_1^2 + (az_2)^2 -1$ for $z \in \R^2$. We define a rotation about $\theta \in [0, 2\pi)$ by $R(\theta)=\begin{pmatrix}
  \cos\theta & -\sin\theta \\
  \sin \theta & \cos\theta
\end{pmatrix}$. For $\theta \in [0, 2\pi)$ and a translation $z_c \in [-1,1]^2$, we then consider the ellipse
\begin{equation}
  E(\theta, z_c) = \left\{z \in \R^2\,:\,f_{\theta, z_c}(z)\coloneqq h_a\left(\tfrac 1 s R\left(\tfrac \theta 2 \right)(z-z_c)\right) \leq 0 \right\}.
\end{equation}
The underlying hidden manifold of these ellipses is $[0, 2\pi)/\!\sim \times [-1, 1]^2 \cong S^1 \times [-1,1]^2$, where $\sim$ denotes identification of $0$ and $2\pi$. The corresponding squared distance between $x=(\theta_x, z_{c,x})$ and $y=(\theta_y, z_{c,y})$ is
\begin{equation*}
{d_M}(x,y)^2 = \min(\abs{\theta_x-\theta_y}, \abs{\theta_x+2\pi-\theta_y}, \abs{\theta_x-2\pi-\theta_y})^2 + \abs{z_{c,x}-z_{c,y}}^2.
\end{equation*}
The ellipses are discretized as $64\times 64$ images by evaluating either the characteristic function of $E(\theta, x_c)$ or the smoothed variant $$\hat f_{\theta, x_c; k}(z) =(1+\exp(kf_{\theta, x_c}))^{-1}$$ on a grid on $[-1,1]^2$. We use a modified version of sampling strategy \ref{S1}, where we sample on the disc with radius $\frac{1}{\sqrt 3}\eps$ taken with respect to the maximum metric.
%%%%%%%%%%%%%%%%%%%%%%%%%%%%

\subsection{Autoencoder architecture and training procedure}\label{subsec:autoencoder}
The used architecture is a deep convolutional neural network as in \cite{BeRaRoGo18}, however, we use larger input images and a smaller latent space dimension. Both encoder and decoder consist of blocks of two convolutions, each followed by a nonlinearity, and a subsequent average pooling (encoder) or upsampling (decoder). 
We determine the number of blocks such that the final output image of the encoder consists of $4\times 4$ pixels, resulting in a latent code of size $16$. The full architectures are shown in \cref{table:architecture}. We use Kaiming initialization \cite{HeZhReSu15}, \ie, all convolutional weights are initialized as zero-mean Gaussian random variables with standard deviation $\sqrt{2} / \sqrt{\texttt{fan\char`_in}(1+0.01^2)}$ for \texttt{fan\char`_in} the number of input channels of the layer times the number of entries in the filter, and all biases are initialized with zeros.
For training, we use the Adam optimizer \cite{KiBa15} with learning rate $10^{-4}$ and default values $\beta_1=0.9, \beta_2=0.999$ and $\varepsilon=10^{-8}$. We further employ a so-called weight decay 
(an $L^2$-penalty on the weights) of $10^{-5}$.
In \cite{BrRaRu21}, we use the non-differentiable leaky ReLU activation function $\lrelu_\alpha(x)=\max\{x,\alpha x\}$. To satisfy the assumption \ref{H2}, one may replace this activation function with a smooth approximation $\lrelu_{\alpha, \beta}(x)=\frac 1 \beta \log(\exp(\beta x)+\exp(\alpha \beta x))$, which converges to $\lrelu_\alpha$ for $\beta \to \infty$. Since the leaky ReLU is much faster to compute and we did not observe any significant effect of the smoothing in practice, we keep the leaky ReLU activation function. 
Networks with (leaky) ReLU nonlinearity also provide nice approximation properties with provable growth rate of the network architecture, though only in Sobolev spaces of order $0\leq s \leq 1$~\cite{GuKuPe20}.
Furthermore, in view of \ref{H3}, the architecture (number of layers and weights, bound on weight growth or penalty on weights) should depend on the sampling radius $\eps$.
However, since our experiments are performed with $\eps$ in a fixed range, we always keep the same architecture.

The training data consists of triplets of images together with a distance value (\cf\ \cref{fig:masterpiece} top): 
\begin{equation*}
  (x,y,\av_{M}(x,y), d_M(x,y)) \text{ with } x,y\in\sample\subset M \times M.
\end{equation*}
This input allows to compute the ingredients  $\discGrad{x}{y}\phi$  and $\discHess{x}{y}\phi$ of the discrete sampling loss functional $E^{\sample}[\phi]$ defined in 
\cref{sec:discreteloss}.  
Even though we generate random samples on the fly, we simulate epochs of size $10000$, split into batches of size $128$, to make the training graphs look more familiar and less noisy.

Our method allows us to train the encoder map separately from the decoder by minimizing 
$E^{\sample}[\phi_\theta]$ to obtain some close to optimal $\phi_{\theta^*}$. Subsequently, if required, one can fix $\theta^*$ and train the decoder by minimizing the \emph{reconstruction loss}
\begin{equation}
R[\psi_\xi]=R[\phi_{\theta^*},\psi_\xi]=\frac1{2|\sample|}\sum_{(x,y)\in\sample}\|\psi_\xi(\phi_{\theta^*}(x))-x\|_{L^2}^2+\|\psi_\xi(\phi_{\theta^*}(y))-y\|_{L^2}^2,\label{eq:rec-loss}
\end{equation}
where for a pixel image $u\in\R^{N\times N\times C}$ with $N\times N$ pixels and $C$ color channels we define the discrete $L^2$-norm as
$\|u\|_{L^2}^2\coloneqq\frac{1}{N^2C}\sum_{i,j=1}^N\sum_{c=1}^Cu_{ijc}^2$.
Alternatively, encoder and decoder can be trained simultaneously by minimizing 
the loss functional
\begin{equation*}
E^{\sample}[\phi_\theta]+\kappa R[\phi_\theta,\psi_\xi]
\end{equation*}
for $\phi_\theta$ and $\psi_\xi$ 
and some $\kappa>0$. 
From now on, we omit the parameters $\theta$ and $\xi$ again, and write $\phi, \psi$ for encoder and decoder.
For better numerical stability in the computation of the finite differences $\discGrad{x}{y}\phi$ and $\discHess{x}{y}\phi$, we rejected pairs with distance below a certain threshold. 
In all our experiments, the isometry, flatness and reconstruction parts of the loss function decrease continuously and monotonically (up to the usual stochastic variations) during training, as shown in \cref{fig:training} for the dataset \ref{dataset:r}.

\begin{figure}[htb]
\begin{tikzpicture}

\pgfplotsset{compat=1.3}

\pgfplotsset{
  short legend/.style={%
    legend image code/.code={
      \draw[##1,line width=0.6pt]
        plot coordinates {
          (0cm,0cm)
          (0.3cm,0cm)
          };%
        }
    }
}

\begin{axis}
[
name=plot1,
width=5.2cm, 
height=3.2cm,
axis background/.style={fill=color0},
axis line style={white},
x grid style={white},
xmajorgrids,
xtick style={draw=none},
xtick = {0, 50, 100, 150, 200, 250, 300, 350},
xticklabels = {0, , , , , , , 350},
xmin=-20,
xmax=360,
xtick style={color=white!15!black},
y grid style={white},
ymajorgrids,
ytick style={draw=none},
ymode = log,
ytick style={color=white!15!black},
x label style={at={(axis description cs:0.5,-0.1)},anchor=north},
xlabel=steps,
ylabel=isometry,
]
\scriptsize
\addplot [semithick, cb-vividorange]
       table[x=Step,y=Value,col sep=comma] {images_cows_e0.25_f10_tangential_run-grad.csv}; 
\addplot [semithick, cb-darkblue]
       table[x=Step,y=Value,col sep=comma] {images_cows_e0.25_f0_tangential_run-grad.csv}; 
\pgfplotsset{scaled x ticks=false}
\end{axis}
\hfill
\begin{axis}
[
name=plot2,
at=(plot1.right of south east), anchor=left of south west,
xshift=0.5cm,
width=5.2cm, 
height=3.2cm,
axis background/.style={fill=color0},
axis line style={white},
x grid style={white},
xmajorgrids,
xtick style={draw=none},
xtick = {0, 50, 100, 150, 200, 250, 300, 350},
xticklabels = {0, , , , , , , 350},
xmin=-20,
xmax=360,
xtick style={color=white!15!black},
y grid style={white},
ymajorgrids,
ytick style={draw=none},
ymode = log,
ytick style={color=white!15!black},
x label style={at={(axis description cs:0.5,-0.1)},anchor=north},
xlabel=steps,
ylabel=bending,
]
\scriptsize
\addplot [semithick, cb-vividorange]
       table[x=Step,y=Value,col sep=comma] {images_cows_e0.25_f10_tangential_run-hess.csv}; 
\addplot [semithick, cb-darkblue]
       table[x=Step,y=Value,col sep=comma] {images_cows_e0.25_f0_tangential_run-hess.csv}; 
\pgfplotsset{scaled x ticks=false}
\end{axis}
\hfill
\begin{axis}
[
at=(plot2.right of south east), anchor=left of south west,
xshift=0.5cm,
width=5.2cm, 
height=3.2cm,
axis background/.style={fill=color0},
legend style={short legend},
axis line style={white},
legend cell align={left},
legend style={
  fill opacity=0.6,
  draw opacity=1,
  text opacity=1,
  at={(0.98,0.98)},
  anchor=north east,
  draw=white!80!black,
  fill=color0
},
x grid style={white},
xmajorgrids,
xtick style={draw=none},
xtick = {0, 250, 500, 750, 1000, 1250, 1500},
xticklabels = {0, , , , , , 1500},
xmin=-20,
xmax=1550,
xtick style={color=white!15!black},
y grid style={white},
ymajorgrids,
ytick style={draw=none},
ymin=0, ymax=0.012,
ytick = {0.002, 0.003, 0.004, 0.005, 0.006, 0.007, 0.008, 0.009, 0.01},
yticklabels = {0.002, , , 0.005, , , , , 0.01},
ymode = log,
ytick style={color=white!15!black},
x label style={at={(axis description cs:0.5,-0.1)},anchor=north},
xlabel=steps,
ylabel=reconstruction,
]
\scriptsize
\addplot [semithick, cb-darkblue]
       table[x=Step,y=Value,col sep=comma] {images_cows_e0.25_f0_tangential_run-rec.csv}; 
\addlegendentry{$\lambda=0$}
\addplot [semithick, cb-vividorange]
       table[x=Step,y=Value,col sep=comma] {images_cows_e0.25_f10_tangential_run-rec.csv}; 
\addlegendentry{$\lambda=10$}
\pgfplotsset{scaled x ticks=false}
\end{axis}
\end{tikzpicture}
\vspace*{-5mm}
\caption{
Comparison of temporal evolution of the three loss components for dataset \ref{dataset:r}  for $\lambda=10$ and $\lambda=0$ (logarithmic $y$-axis).
Per optimization step, 10000 images are processed in batches of 128.}
\label{fig:training}
\end{figure}
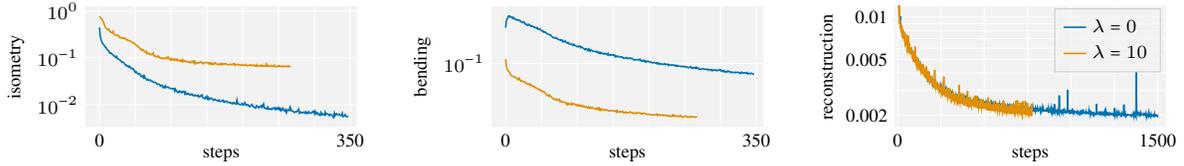

\begin{table}
{\small
\begin{tabular}[t]{@{}lll@{}}
type               & parameters & output size  \\ \midrule
Conv2d             & $ 1 / 1 / 1 / 16$    & $66 \times  66 \times  16$ \\ \midrule
Conv2d + LeakyReLU & $ 3 / 1 / 1 / 16$    & $66 \times  66 \times  16$ \\
Conv2d + LeakyReLU & $ 3 / 1 / 1 / 16$    & $66\times 66\times 16$     \\
AvgPool2d          & $ 2/2$               & $33\times 33\times 16$     \\ \midrule
Conv2d + LeakyReLU & $3 / 1 / 1 / 32 $    & $33 \times  33 \times  32$ \\
Conv2d + LeakyReLU & $3 / 1 / 1 / 32 $    & $33 \times  33 \times  32$ \\
AvgPool2d          & $ 2/2$               & $16\times 16\times 32$     \\ \midrule
Conv2d + LeakyReLU & $3 / 1 / 1 / 64 $    & $16\times 16\times 64$     \\
Conv2d + LeakyReLU &$3 / 1 / 1 / 64 $     & $16\times 16\times 64$     \\
AvgPool2d          & $ 2/2$               &$ 8\times 8\times 64$       \\ \midrule
Conv2d + LeakyReLU & $3 / 1 / 1 / 128 $   & $8\times 8\times 128$      \\
Conv2d + LeakyReLU & $3 / 1 / 1 / 128 $   & $8\times 8\times 128$      \\
AvgPool2d          & $ 2/2$               & $4\times 4\times 128$      \\ \midrule
Conv2d + LeakyReLU & $3 / 1 / 1 / 256 $   & $4\times 4\times 256$      \\
Conv2d             & $3 / 1 / 1 / 1 $     & $4\times 4\times 1$      
\end{tabular} \hspace{1cm}
\begin{tabular}[t]{@{}lll@{}}
type               & parameters & output size  \\ \midrule
Conv2d + LeakyReLU & $ 3 / 1 / 1 / 128$   & $4 \times  4 \times  128$ \\
Conv2d + LeakyReLU & $ 3 / 1 / 1 / 128$   & $4\times 4\times 128$     \\
Upsample          & $ 2$ / nearest         & $8\times 8\times 128$     \\ \midrule
Conv2d + LeakyReLU & $3 / 1 / 1 / 64 $    & $8 \times  8 \times  64$ \\
Conv2d + LeakyReLU & $3 / 1 / 1 / 64 $    & $8 \times  8 \times  64$ \\
Upsample          & $ 2$ / nearest         & $16\times 16\times 64$     \\ \midrule
Conv2d + LeakyReLU & $3 / 1 / 1 / 32 $    & $16\times 16\times 32$     \\
Conv2d + LeakyReLU &$3 / 1 / 1 / 32 $     & $16\times 16\times 32$     \\
Upsample          & $ 2$ / nearest         &$ 32\times 32\times 32$       \\ \midrule
Conv2d + LeakyReLU & $3 / 1 / 1 / 16 $    & $32\times 32\times 16$      \\
Conv2d + LeakyReLU & $3 / 1 / 1 / 16 $    & $32\times 32\times 16$      \\
Upsample          & $ 2$ / nearest         & $64\times 64\times 16$      \\ \midrule
Conv2d + LeakyReLU & $3 / 1 / 1 / 16 $    & $64\times 64\times 16$      \\
Conv2d             & $3 / 1 / 1 / 1 $     & $64\times 64\times 1$      
\end{tabular}
\vspace{0.4cm}}
\caption{\label{table:architecture}Encoder (left) and decoder (right) architectures for $64\times 64$ grayscale images. Parameters for convolutional layers (Conv2d) are filter size / stride / padding / output channels, where filter size, stride and padding are in each direction. Parameters for average pooling layers (AvgPool2d) are filter size and stride, again in each direction. Parameters for upsampling layers (Upsample) are scale and mode. For the LeakyReLU activation, a negative slope of $\alpha=0.01$ was used. Output sizes are given with channels in the last dimension. The architecture for $128 \times 128$ $3$-channel color images is similar, except that there is one more block, thus the  maximum number of output channels of the encoder is $512$.}
\end{table}

\subsection{Visualization of the embeddings}\label{subsec:visualization}
To evaluate the smoothness of our embedding function, we visualize the embedded manifold $\phi(M)$ by performing a  principal component analysis 
\begin{figure}
\centering
\input{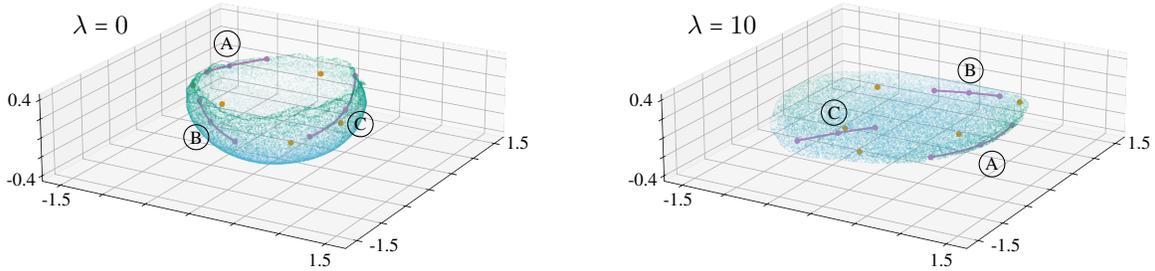}
\caption{\label{fig:sundials-lambda}
Comparison of the latent manifold $\phi(M) \subset \R^{16}$ for sundial dataset \ref{dataset:s} for $\lambda=0$ and $\lambda=10$ and $\eps=\pi/8$ (colored points as in \cref{fig:masterpiece}). The same training procedure as in \cref{fig:masterpiece} was used. The three visualized components explain $97.8\%$ $(\lambda = 0)$ and $99.95\%$ $(\lambda = 10)$ of the variance. A threshold of $99\%$ variance is reached for $6$ $(\lambda=0$) and $2$ $(\lambda=10)$ components.}
\end{figure}
(PCA) in latent space. We then visualize projections onto three-dimensional subspaces 
spanned by three selected principal directions of the PCA.
In all cases we observe smooth embeddings that neatly reproduce the geometry and topology of the hidden manifold $M$ (see \cref{fig:masterpiece,fig:sundials-lambda,fig:arcs-summary,fig:cows-qualitative,fig:quantization}). 
Note that self-intersections in the three-dimensional projection do not necessarily reflect self-intersections in the full latent spaces (see \cref{fig:arcs-distances,fig:arcs-summary}).
For dataset \ref{dataset:s}, our approach and its result are summarized in \cref{fig:masterpiece} for $\lambda=0$. The geometry of the hemisphere is reproduced. In comparison, \cref{fig:sundials-lambda} shows the result for a higher bending penalty $\lambda=10$, resulting in a much flatter embedding of the hemisphere (\ie, the variance in the third component is significantly smaller). 

\begin{figure}
\scalebox{0.9}{
\input{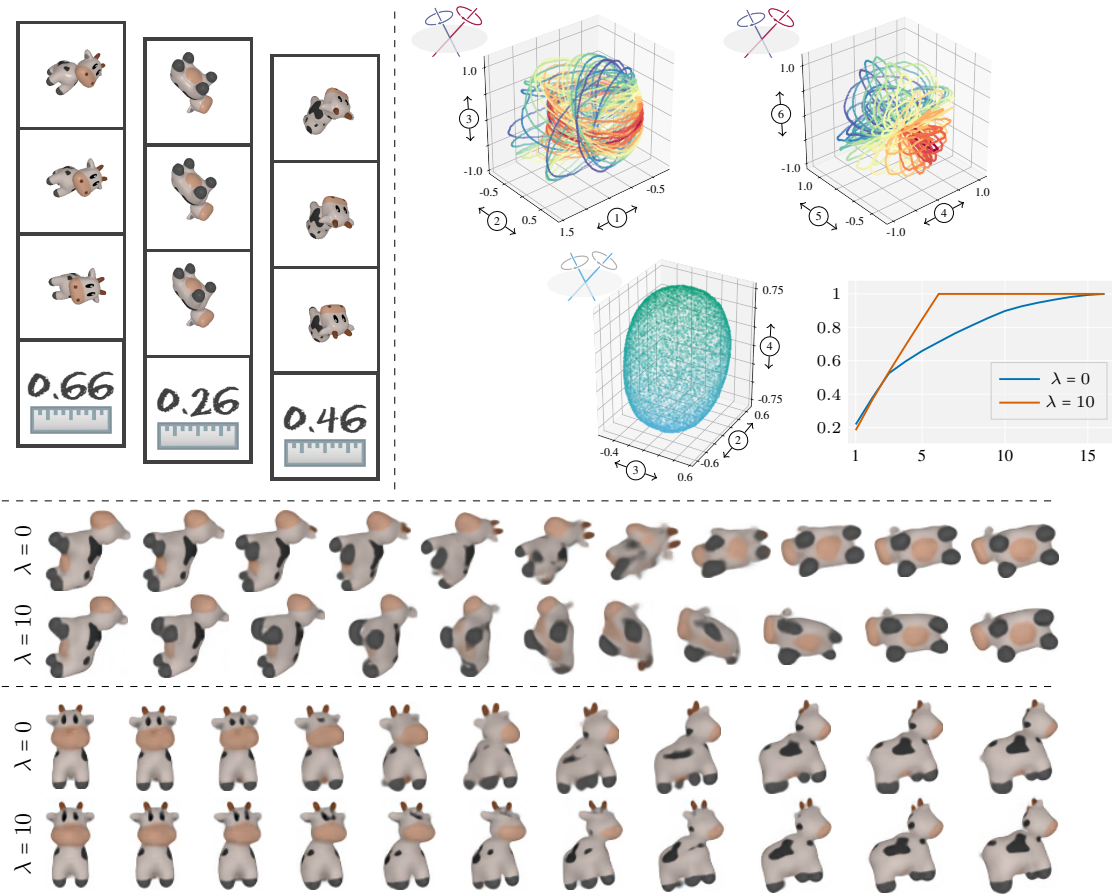}}
\caption{\label{fig:cows-qualitative} Visualization of the results for dataset \ref{dataset:r} for $\lambda=10$ and $\eps=\frac{\pi}{4}$ (with maximal distance being $\pi$). Encoder and decoder were trained separately. Training of the encoder was stopped after the value of the training loss did not decrease for 20 epochs. The decoder was trained until the reconstruction error $R$ evaluated on a test set reached a threshold of $2\cdot 10^{-3}$. Pairs with distance below $\frac{1}{20}$ of the maximal distance were rejected. The embedding is smooth, revealing the topology and geometry of the manifold $SO(3)$, and reasonable interpolations are obtained for $\lambda=10$. On the left, three input triples and their distances  are shown. To visualize the latent manifold, a PCA was performed on the point cloud in 16-dimensional latent space obtained by applying the embedding map $\phi$ to a large number of images from $M\cong SO(3)$. For the bottom right image we fixed a rotation angle and randomly sampled the rotation axis from $S^2$. For both other visualizations we regularly sample rotation axes from different latitudes of $S^2$ and then show points corresponding to rotations around each axis in the same color. The graph coordinates correspond to the principal components indicated by their labels. In components 2, 3, 4, a sphere-like structure can be observed, suggesting that these three components encode the rotation axis. The graph on the bottom right shows the the total amount of explained variance as a function of increasing subspace dimension, with a threshold of 99\% reached for 6 dimensions for $\lambda=10$. Below the dashed line we show examples of interpolations generated by interpolating linearly in latent space and subsequent decoding. For vanishing bending regularization $\lambda=0$ those interpolations are unreliable, while they look very reasonable for $\lambda=10$, even though the decoder was neither trained on linear interpolations in latent space nor was it regularized. }
\end{figure}

%fffffffff
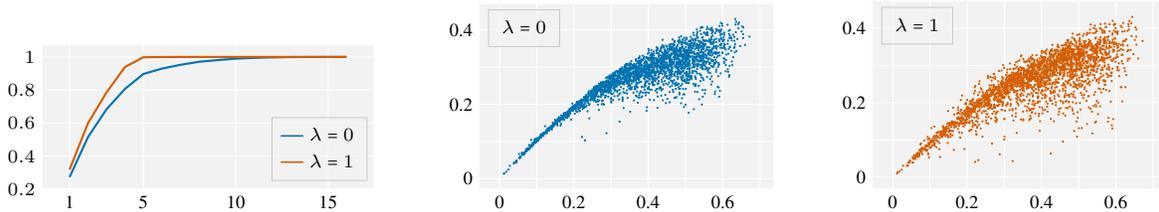
\begin{figure}
\begin{tikzpicture}

\pgfplotsset{
  short legend/.style={%
    legend image code/.code={
      \draw[##1,line width=0.6pt]
        plot coordinates {
          (0cm,0cm)
          (0.3cm,0cm)
          };%
        }
    }
}

\begin{axis}
[
width=6cm, 
height=3.5cm,
axis background/.style={fill=color0},
axis line style={white},
legend style={short legend},
x grid style={white},
xmajorgrids,
xtick style={draw=none},
legend style={
  fill opacity=0.6,
  draw opacity=1,
  text opacity=1,
  draw=white!80!black,
  fill=color0,
  at={(0.98,0.5)},
  anchor=north east,
},
xtick = {1, 5, 10, 15},
xticklabels = {1, 5, 10 ,15},
xtick style={color=white!15!black},
y grid style={white},
ymajorgrids,
ytick style={draw=none},
yticklabel style={/pgf/number format/.cd,fixed,precision=3},
ytick style={color=white!15!black},
]
\scriptsize
\addplot [thick, cb-darkblue]
       table[x expr = \thisrow{c}+1, y=var,col sep=comma] {images_klein_bottle_e0.25_f0_rejection_exp_var_sum.csv}; 
\addlegendentry{$\lambda=0$};
\addplot [thick, cb-strongorange]
       table[x expr = \thisrow{c}+1, y=var,col sep=comma] {images_klein_bottle_e0.25_f1_rejection_exp_var_sum.csv}; 
\addlegendentry{$\lambda=1$};
\pgfplotsset{scaled x ticks=false}
\end{axis}
\end{tikzpicture} \hfill
\begin{tikzpicture}
\begin{axis}
[
width=5.5cm, 
axis background/.style={fill=color0},
axis line style={white},
legend style={
  fill opacity=0.6,
  draw opacity=1,
  text opacity=1,
  draw=white!80!black,
  fill=color0,
},
legend pos=north west,
x grid style={white},
xmajorgrids,
xtick style={draw=none},
axis equal image,
xtick = {0, 0.1, 0.2, 0.3, 0.4, 0.5, 0.6,0.7},
xticklabels = {0, , 0.2, , 0.4, , 0.6,},
xtick style={color=white!15!black},
y grid style={white},
ymajorgrids,
ytick={0, 0.1, 0.2, 0.3, 0.4},
yticklabels={0, , 0.2,,0.4},
ytick style={draw=none},
ytick style={color=white!15!black},
]
\addlegendimage{empty legend}
\addlegendentry{$\lambda=0$};
\scriptsize
\addplot+[mark=*, mark options={color=cb-darkblue, scale=0.1pt}, only marks] table[x=distances_true, y=distances_latent, col sep=comma] {images_klein_bottle_e0.25_f0_rejection_scatter_distances.csv};
\pgfplotsset{scaled x ticks=false}
\end{axis}
\end{tikzpicture}
\hfill
\begin{tikzpicture}
\begin{axis}
[
width=5.5cm, 
axis background/.style={fill=color0},
axis line style={white},
x grid style={white},
xmajorgrids,
xtick style={draw=none},
legend style={
  fill opacity=0.6,
  draw opacity=1,
  text opacity=1,
  draw=white!80!black,
  fill=color0,
},
legend pos=north west,
axis equal image,
xtick = {0, 0.1, 0.2, 0.3, 0.4, 0.5, 0.6,0.7},
xticklabels = {0, , 0.2, , 0.4, , 0.6,},
xtick style={color=white!15!black},
y grid style={white},
ymajorgrids,
ytick={0, 0.1, 0.2, 0.3, 0.4},
yticklabels={0, , 0.2,,0.4},
ytick style={draw=none},
ytick style={color=white!15!black},
]
\scriptsize
\addlegendimage{empty legend}
\addlegendentry{$\lambda=1$};
\addplot+[mark=*, mark options={color=cb-strongorange, scale=0.1pt}, only marks] table[x=distances_true, y=distances_latent, col sep=comma] {images_klein_bottle_e0.25_f1_rejection_scatter_distances.csv}; 
\pgfplotsset{scaled x ticks=false}
\end{axis}
\end{tikzpicture}
\caption{\label{fig:arcs-distances}  Comparison of two experiments for the dataset \ref{dataset:a} for $\lambda=0$ and $\lambda=1$ and $\eps=\frac{1}{4\sqrt 2}$ (with maximal distance being $\frac{1}{\sqrt 2}$). Encoder and decoder were trained at the same time with $\kappa=0.1$ and training was stopped once the reconstruction error $R$ evaluated on a test set reached a threshold of $2\cdot 10^{-3}$. Pairs with distance below $\frac{1}{20}$ of the maximal distance were rejected. 
The left graph shows the total amount of explained variance as a function of increasing subspace dimension, the total dimension of the latent space being 16. The threshold of $99\%$ explained variance was reached for 5 and 11 dimensions for $\lambda=1$ and $\lambda=0$, respectively. 
The right graphs show the extrinsic latent space distance $|\phi(x)-\phi(y)|$ versus the intrinsic manifold distance $d_M(x,y)$ for random point pairs $(x,y)\in M$. }
\end{figure}
%fffffffff

%fffffffff
\begin{figure}
\input{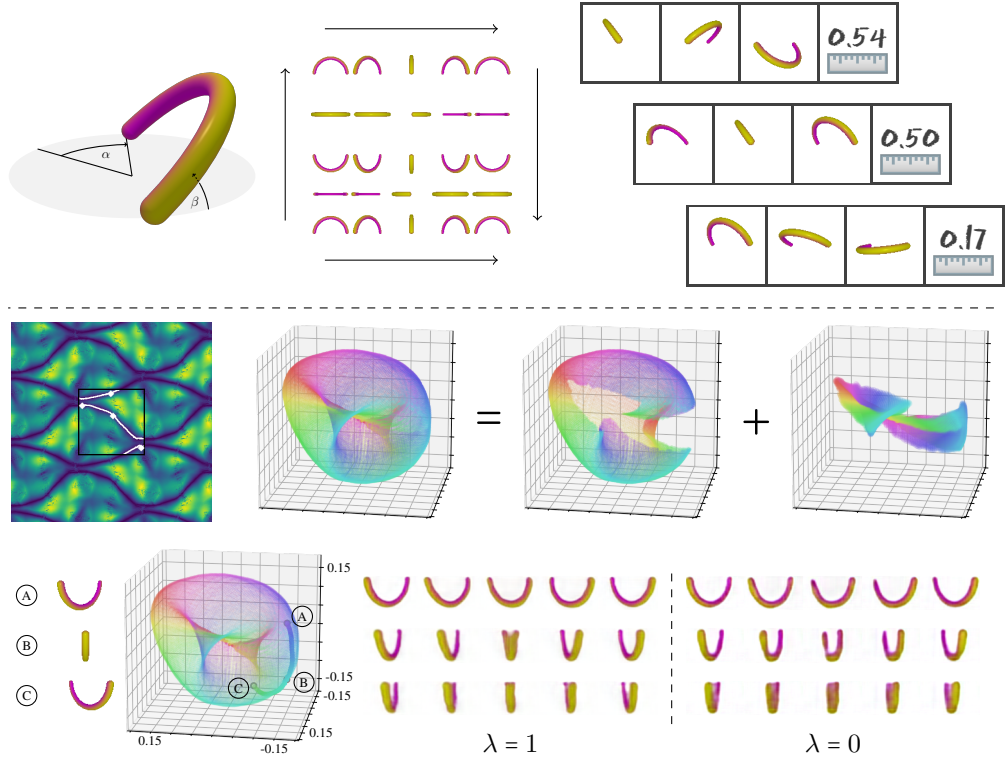}
\caption{
Visualization of the results for dataset \ref{dataset:a} for $\lambda=1$ (unless indicated otherwise). The same experiments as in \cref{fig:arcs-distances} were used.
The first row illustrates how the data was generated: The first image shows how an arc is rotated by angles $\alpha$ and $\beta$. The second image shows variation of $\alpha$ in horizontal and variation of $\beta$ in vertical direction, making visible the mirroring leading to the topology of the Klein bottle. 
On the right, three input triples and their distances are shown. The second row explains the latent manifold. 
First, a point cloud in 16-dimensional latent space is generated by applying the embedding map $\phi$ to a set of images obtained from a regular grid on $[0,1]^2$. This point cloud is then projected onto the first three components determined by a PCA. The first image visualizes the quantity \eqref{eq:self-intersection} evaluated on this grid in the black frame. 
The image is extended in all directions according to the identification of the sides. The white lines were computed using the watershed algorithm. Each one of the lines (circles) is folded together, with the diamond markers showing the folding points. The folding becomes visible in the third image. Points are colored according to the original $\alpha$ coordinate, using a cyclic color map. 
The third row shows the embedding of a geodesic curve between $A$ and $C$, where $B$ is their geodesic midpoint, as well as the images obtained through linear interpolation in latent space for $\lambda=0$ and $\lambda=1$ using different starting points symmetric around $B$. 
\label{fig:arcs-summary}} 
\end{figure}
%fffffffff

In all our experiments, the latent space dimensionality $l$ is chosen substantially larger than the minimum value required for a smooth embedding -- in real applications the intrinsic dimensionality $m$ is unknown beforehand anyway, so that one needs to pick $l$ rather larger than smaller.  Furthermore, the encoder might achieve flatter embeddings when more dimensions are available.
For instance, $S^1\times S^1\subset\R^4$ provides a better embedding of the flat torus than would be possible in $\R^3$.
We observe that the encoder indeed makes use of that freedom for dataset \ref{dataset:r}:
The bottom right graph in \cref{fig:cows-qualitative} shows the total amount of explained variance as a function of increasing subspace dimension.
Apparently the embedding uses six Euclidean dimensions,
even though $l=5$ would be enough for a (potentially non-isometric) embedding ($SO(3)\cong\R P^3$, which embeds into $\R^5$, but not $\R^4$ \cite{Ho40,Ha38}).
In \cref{fig:arcs-distances} we can similarly read off the used number of dimensions for dataset \ref{dataset:a}.
Here the situation is similar to, but not as obvious as before:
to the best of our knowledge, it is not known whether there exists a smooth isometric embedding of the Klein bottle into less than five dimensions.
However, isometric immersions (with self-intersections) exist into four dimensions (\cf~\cite{To41}).
Since our encoder loss functional cannot distinguish between immersions and embeddings (there is no term preventing self-intersections),
the minimum dimensionality would thus be no larger than four, but our latent manifold actually uses five Euclidean dimensions in the case $\lambda>0$.
Interestingly, without bending regularization, \ie~for $\lambda=0$, the embedding uses many more dimensions.
As discussed above, this is not because a minimization of the isometry term would make this necessary.
It is rather a manifestation of the strong degeneracy of the set of isometric embeddings, which are possible using arbitrarily many Euclidean dimensions.
\Cref{fig:arcs-summary} illustrates that the topology of the Klein bottle was reproduced for that dataset. It is well-known that no three-dimensional embedding without self-intersections exists. This explains the self-intersections in each of the visualized projections.
In \cref{fig:arcs-summary} we also analyze the nature of those self-intersections occurring after projection: Denote the composition of $\phi$ with the projection onto a 3-dimensional subspace $V$ by $\phisub$. Then such a self-intersection after projection is characterized by $\phisub(x)=\phisub(y)$ for $x \neq y$. Thus, a suitable indicator of a self-intersection at $x$ is
\begin{equation}\label{eq:self-intersection}
  \min_{y \in M, y \neq x}\frac{\abs{\phisub(x)-\phisub(y)}}{d_M(x,y)}.
\end{equation}
If this quantity is small, we expect a self-intersection at $x$. 
Recall that, when training the encoder separately, our method in principle does not prevent self-intersections. In \cref{fig:arcs-distances} we plot true distances on the Klein bottle against Euclidean distances of the corresponding embedded points in latent space. The resulting graph clearly indicates the absence of any self-intersections. While this might be due to the included decoder, which promotes injectivity of the embedding, we observed similar results when training the encoder separately. This plot further shows that, for $\lambda=0$, the distances are more or less preserved up to the $\eps$-value used in the training.Instead, for $\lambda=1$, the isometry of the embedding is traded off for less bending.

%%%%%%%%%%%%%%%%%%%%%%%%%%%%%%%%%%%%%%%%%%%%%%%%%%%%%%%%%%%

\subsection{Linear interpolation in latent space}\label{subsec:interpolation}
Let $x, y$ be two input images. We may obtain an interpolant between $x$ and $y$ in the following way.  Let $\tilde x$ and $\tilde y\in \R^l$ be the latent codes corresponding to $x, y$ and let $\tilde z$ be a linear interpolation of $\tilde x, \tilde y$. Decoding $\tilde z$ gives the desired interpolant $z$ between $x$ and $y$. In formulas, $z=\psi(\alpha \phi(x) + (1-\alpha) \phi(y))$ for an interpolation coefficient $\alpha \in [0, 1]$.
Flatness of the latent manifold is expected to improve 
the usefulness of interpolants obtained in this way across moderate distances, since one then expects the interpolant $\tilde z$ to be close to the latent manifold. This is qualitatively 
illustrated in \cref{fig:cows-qualitative} for dataset \ref{dataset:r}. 
For $\lambda=0$, the decoder output of linear interpolations in latent space does not at all reproduce continuously rotating objects, whereas it clearly does for $\lambda=10$.
Since our focus is on the regularizing properties of our encoder loss functional $E^{\sample}$,
the decoder was intentionally not trained on linear interpolation and furthermore not additionally regularized.
An additional training of the decoder on linear interpolants in latent space would likely improve the interpolation quality even more,
and perhaps even allow the decoder to compensate for the deficiencies visible in \cref{fig:cows-qualitative,fig:quantization} for $\lambda=0$.
Due to the extrinsic curvature of the latent manifold, linear interpolation of course only makes sense for sufficiently close endpoints, as illustrated in \cref{fig:quantization}.
To quantify the error of a linear interpolation in latent space as a function of interpolation distance, for a fixed test sample set $\sample[]\subset M\times M$
we let $\sample[\delta]=\{(x,y)\in \sample[]\,:\, d_M(x,y)\leq \delta\}$ and define
\begin{align}\label{eq:error-lat}%
  \err_{\text{lat}}^2(\delta) = \frac{1}{\abs{\sample[\delta]}}\sum_{x,y\in\sample[\delta]}\abs{\phi(\av_M(x,y))-\av_{\R^l}(\phi(x),\phi(y))}^2.
\end{align}
We further measure the $L^2$-error to the ground truth geodesic interpolation in image space using an error measure $\err_{\text{im}}$ defined as
\begin{align}\label{eq:error}
\err_{\text{im}}(\delta)^2 &= {\frac{1}{\abs{\sample[\delta]}}\sum_{(x,y)\in\sample[\delta]} \max\{\err_{\text{i}}(x,y)^2 - \err_{\text{b}}(x,y)^2, 0\}}\text{ for}\\
\err_{\text{i}}(x,y) &= \|\av_M(x,y) - \psi(\av_{\R^l}(\phi(x),\phi(y)))\|_{L^2},\\
\err_{\text{b}}(x,y) &= \|\av_M(x,y) - \psi(\phi(\av_M(x,y)))\|_{L^2}.
\end{align}
Here, $\err_{\text{i}}$ is the error due to linear interpolation, and $\err_{\text{b}}$ is the base reconstruction error, which occurs independently of interpolation. Note that it is expected (but not guaranteed) that $\err_{\text{i}} \geq \err_{\text{b}}$.
%
%fffffffff
\begin{figure}
  \centering
  \begin{tikzpicture}

\pgfplotsset{
  short legend/.style={%
    legend image code/.code={
      \draw[##1,line width=0.6pt]
        plot coordinates {
          (0cm,0cm)
          (0.3cm,0cm)
          };%
        }
    }
}

\begin{axis}
[
name=plot,
title='Rotating Object ({\bf R})',
width=7cm, 
height=4cm,
axis background/.style={fill=color0},
axis line style={white},
legend style={short legend},
x grid style={white},
xmajorgrids,
xtick style={draw=none},
legend style={
  fill opacity=0.6,
  draw opacity=1,
  text opacity=1,
  draw=white!80!black,
  fill=color0,
  at={(0.98,0.5)},
  anchor=north east,
},
xlabel=$\delta$,
ylabel=$\err_{\text{im}}(\delta)$,
xtick style={color=white!15!black},
y grid style={white},
ymajorgrids,
scaled ticks=false,
ytick style={draw=none},
yticklabel style={/pgf/number format/.cd,fixed,precision=3, /tikz/.cd},
ytick style={color=white!15!black},
]
\scriptsize
\addplot [thick, cb-darkblue]
       table[x expr=\thisrow{eps}, y expr=sqrt(\thisrow{0}),col sep=comma] {images_cows_results_compared.csv}; 
\addlegendentry{$\lambda=0$};
\addplot [thick, cb-strongorange]
       table[x expr=\thisrow{eps}, y expr=sqrt(\thisrow{10}),col sep=comma] {images_cows_results_compared.csv}; 
\addlegendentry{$\lambda=10$};
\addplot [thick, cb-green]
       table[x expr=\thisrow{eps}, y expr=sqrt(\thisrow{auto}),col sep=comma] {images_cows_results_compared.csv}; 
\addlegendentry{no regul.};
\end{axis}

\begin{axis}
[
at=(plot.right of south east), anchor=left of south west,
xshift=0.5cm,
title='Arc rotating around two specified axes ({\bf A})',
width=7cm, 
height=4cm,
axis background/.style={fill=color0},
axis line style={white},
legend style={short legend},
x grid style={white},
xmajorgrids,
xtick style={draw=none},
legend style={
  fill opacity=0.6,
  draw opacity=1,
  text opacity=1,
  draw=white!80!black,
  fill=color0,
  at={(0.98,0.5)},
  anchor=north east,
},
xlabel=$\delta$,
ylabel=$\err_{\text{im}}(\delta)$,
xtick style={color=white!15!black},
y grid style={white},
ymajorgrids,
scaled ticks=false,
ytick style={draw=none},
yticklabel style={/pgf/number format/.cd,fixed,precision=3, /tikz/.cd},
ytick style={color=white!15!black},
]
\scriptsize
\addplot [thick, cb-darkblue]
       table[x=eps, y expr=sqrt(\thisrow{0}),col sep=comma] {images_klein_bottle_results_compared.csv}; 
\addlegendentry{$\lambda=0$};
\addplot [thick, cb-strongorange]
       table[x=eps, y expr=sqrt(\thisrow{1}),col sep=comma] {images_klein_bottle_results_compared.csv}; 
\addlegendentry{$\lambda=1$};
\addplot [thick, cb-green]
       table[x=eps, y expr=sqrt(\thisrow{auto}),col sep=comma] {images_klein_bottle_results_compared.csv}; 
\addlegendentry{no regul.};
\end{axis}
\end{tikzpicture}
  \caption{\label{fig:quantitative} Comparison of the interpolation error for different regularizations for datasets \ref{dataset:r} (left) and \ref{dataset:a} (right).The same training procedures as stated in \cref{fig:cows-qualitative} (for \ref{dataset:r}) and \cref{fig:arcs-distances} (for \ref{dataset:a}) were used.
  }f
\end{figure}
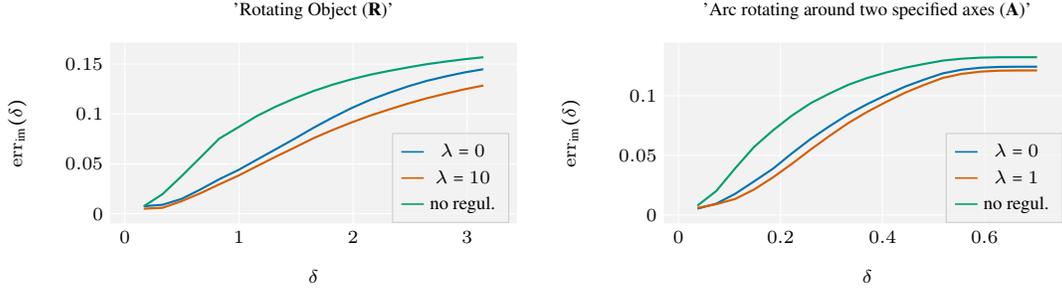
%fffffffff
%
For datasets \ref{dataset:a} and \ref{dataset:r}, \cref{fig:quantitative} compares $\err_{\text{im}}(\delta)$ for different regularizations.
The error is highest without encoder regularization, and it is in particular already high even for small distances.
Our regularization reduces the error, particularly for small interpolation distances.
In our controlled environment it is possible to also test the interplay of sampling radius and bending parameter $\lambda$. 
%fffffffff
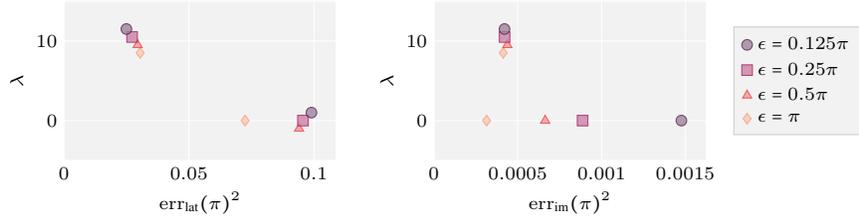
\begin{figure}
\centering
  \begin{tikzpicture}

\pgfplotsset{compat=1.3}

\definecolor{color1}{rgb}{0.29977678, 0.11356089, 0.29254823}
\definecolor{color2}{rgb}{0.63139686, 0.10067417, 0.35664819}
\definecolor{color3}{rgb}{0.90848638, 0.24568473, 0.24598324}
\definecolor{color4}{rgb}{0.96298491, 0.6126247, 0.45145074}

\begin{axis}
[
name=plot1,
width=5.2cm, 
height=3.7cm,
axis background/.style={fill=color0},
axis line style={white},
ylabel=$\lambda$,
xlabel=$\err_{\text{lat}}(\pi)^2$,
y grid style={white},
ymajorgrids,
ytick style={draw=none},
ytick = {0, 10},
ymin=-5,
ymax=15,
ytick style={color=white!15!black},
scaled ticks=false,
xticklabel style={/pgf/number format/.cd,fixed,precision=4, 
/tikz/.cd},
x grid style={white},
xmin=0,
xmajorgrids,
xtick style={draw=none},
xtick style={color=white!15!black},
legend pos=outer north east,
legend cell align=left,
legend style={
  fill opacity=1,
  draw opacity=1,
  text opacity=1,
  at={(1.1,0.5)},
  anchor=west,
  draw=white!80!black,
  fill=color0
},
scatter/classes={%
              0.125={mark=*, fill=color1!55!white, draw=color1},
              0.25={mark=square*,fill=color2!55!white, draw=color2},%
              0.5={mark=triangle*,fill=color3!55!white, draw=color3},%
              1={mark=diamond*, fill=color4!55!white, draw=color4}}
]
\scriptsize
\addplot+[scatter, only marks, opacity=0.8, mark options={scale=1}, scatter src=explicit symbolic] table[x expr=\thisrow{val}, y=lambda, col sep=comma, meta=eps] {images_sundials_mean_latent.csv};
\end{axis}
\begin{axis}
[
name=plot2,
at=(plot1.right of south east), anchor=left of south west,
width=5.2cm, 
height=3.7cm,
axis background/.style={fill=color0},
axis line style={white},
xshift=0.5cm,
ylabel=$\lambda$,
xlabel=$\err_{\text{im}}(\pi)^2$,
y grid style={white},
ymajorgrids,
ytick style={draw=none},
ytick = {0, 10},
ymin=-5,
ymax=15,
ytick style={color=white!15!black},
scaled ticks=false,
xticklabel style={/pgf/number format/.cd,fixed,precision=4, 
/tikz/.cd},
x grid style={white},
xmin=0,
xmajorgrids,
xtick style={draw=none},
xtick style={color=white!15!black},
legend pos=outer north east,
legend cell align=left,
legend style={
  fill opacity=1,
  draw opacity=1,
  text opacity=1,
  at={(1.1,0.5)},
  anchor=west,
  draw=white!80!black,
  fill=color0
},
scatter/classes={%
              0.125={mark=*, fill=color1!55!white, draw=color1},
              0.25={mark=square*,fill=color2!55!white, draw=color2},%
              0.5={mark=triangle*,fill=color3!55!white, draw=color3},%
              1={mark=diamond*, fill=color4!55!white, draw=color4}}
]
\scriptsize
\addplot+[scatter, only marks, opacity=0.8, mark options={scale=1}, scatter src=explicit symbolic] table[x expr=\thisrow{val}, y=lambda, col sep=comma, meta=eps] {images_sundials_mean.csv};
\legend{$\epsilon=0.125\pi$,$\epsilon=0.25\pi$,$\epsilon=0.5\pi$,$\epsilon=\pi$}
\end{axis}
\end{tikzpicture}
    \caption{\label{fig:sundials-compare-eps} Average interpolation errors in latent and image space for training with different sampling radii $\eps$ computed across a test set consisting of random pairs with arbitrary distance on the upper hemisphere. The encoder was trained separately from the decoder, and the training was stopped after the value of the functional $E^{\sample}[\phi]$ evaluated on a different test set did not decrease for 20 epochs. Pairs with distance below $\frac{1}{100}$ of the maximal distance were rejected. The decoder was trained until an accuracy of $10^{-5}$ was reached on the test set. This procedure was repeated three times with different random seeds. Note that the points were slightly shifted in y direction for better distinction.}
\end{figure} 
%fffffffff
Recall that we use the same network architecture and penalty on the network weights for all $\eps$, since we only consider a fixed range of $\epsilon$ values anyway.
In \cref{fig:sundials-compare-eps}, we evaluated the interpolation {errors $\err_{\text{lat}}(\pi)$ and $\err_{\text{im}}(\pi)$ from \eqref{eq:error-lat} and \eqref{eq:error}} for dataset \ref{dataset:s} on a test set of arbitrarily chosen point pairs using sampling strategy \ref{S3}.
{For vanishing bending parameter $\lambda$ we observe a dependence of the interpolation error on the sampling radius $\eps$:
A larger value of $\eps$ is associated with a smaller interpolation error, indicating that a large sampling radius $\eps$ also has some regularizing effect.
A positive bending parameter $\lambda$ more or less overrides the influence of the sampling radius:
The interpolation error gets consistently small, independent of $\eps$.}
This confirms our hypothesis that the bending penalty produces latent manifolds with a smoother structure. This, in turn,  is expected to reduce the generalization error in postprocessing tasks (indeed, the encoder was only trained on point pairs with distance below $\eps$, but the interpolation {errors $\err_{\text{im}}$ and $\err_{\text{lat}}$ are} small for arbitrary point pairs).

%%%%%%%%%%%%%%%%%%%%%%%%%%%%%%%%%%%%%%%%%%%%%%%%%%%%%%
\subsection{Noise in the embedding due to image quantization}\label{subsec:noise}
To illustrate the regularization properties of our loss functional, we did not add any artificial noise to our images, as noisy images would require additional tailored regularization. 
However, the process of generating images from low dimensional manifolds may introduce noise, for instance by discretization as pixel images in datasets \ref{dataset:a} and \ref{dataset:r}. 
We illustrate this effect on the dataset \ref{dataset:e} by comparing smooth and binary images. 
Indeed, while the set of smooth ellipse images is infinite, the set of binary ellipse images is finite due to quantization. 
Thus, while the underlying manifold is still the same, the corresponding set in image space is different. For instance, the same pair of binary ellipse images may be sampled with different distances, leading to a type of noise.
\Cref{fig:quantization} shows that the cylindrical structure of the resulting latent manifold $\phi(M)$ is not as clean in this case. Increasing the resolution of the images would reduce the disparity between the underlying manifold and the image manifold, leading to a reduction of the observed effect.

%fffffffff
\begin{figure}
\input{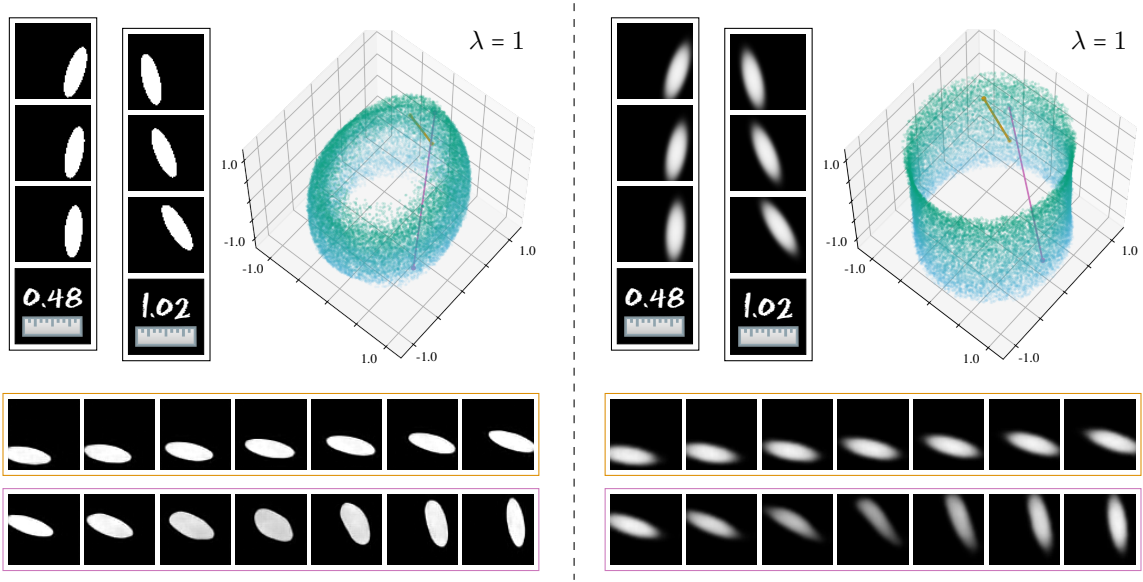}
\caption{Visualization of the results for dataset \ref{dataset:e} for $\lambda=1$ and $\eps\approx 1.06$ (with maximal distance $\approx 4.32$) for different ways of rendering ellipses: as an indicator function (left) and as smoothed indicator function with $k=3$ (right). In each case, two input triples and their distances are shown. PCA coordinates $1,\, 2,\,3$  are used to visualize the latent manifold.  On the bottom, example interpolations are visualized. Their corresponding interpolation path in latent space is shown in the PCA plot in the color of the frame.
Encoder and decoder were trained separately. Training of the encoder was stopped after the value of the training loss did not decrease for 20 epochs. Pairs with distance below $\frac{1}{100}$ of the maximal distance were rejected. The decoder was trained until the reconstruction error $R$ evaluated on a test set reached a threshold of $3\cdot 10^{-3}$ (binary images) and $10^{-4}$ (smoothed images).
\label{fig:quantization}}
\end{figure}
%fffffffff

%%%%%%%%%%%%%%%%%%%%%%%%%%%%%%%%%%%%%%%%%%%%%%%
\section{Some open questions}
We considered a two-step limit process for the discrete sampling loss functional $E^{\sample}$, where we first let $\vert \sample\vert \to \infty$ and then $\epsilon \to 0$.
It seems plausible that one could also perform both limits simultaneously.
Furthermore, we restricted the space of admissible functions, or rather admissible deep neural networks, according to $\eps$.
As discussed in \cref{subsec:discrete_existence}, this restriction might, at least in certain settings, not be necessary.
The limit energy $\mathcal{E}$ depends on second order derivatives of the encoder mapping and shows some resemblance to elastic energies of thin shells.
By classical spectral arguments for discrete elliptic operators as in \cite{Thomee2006}, explicit gradient descent methods for the minimization of such elastic energies come with severe time step restrictions of the order $h^4$, where $h$ is the spatial discretization width.
It is unclear if the structure of our regularization loss implies similar restrictions on the learning rate (the time step of the stochastic gradient descent) used when training the encoder. 
A simultaneous training of encoder and decoder via a minimization of the (weighted) sum of encoder regularization loss and reconstruction loss
may have additional regularizing effects, since some embeddings may be more favorable for the decoder than others.
Such effects have yet to be understood, as well as whether and how the decoder should eventually be regularized. This of course depends on the application in mind.
For larger values of $\lambda$, the loss functional ensures that linear interpolation in latent space becomes a better approximation to Riemannian interpolation on the latent manifold.
Correspondingly, we have seen a stronger flattening of the embedding in \cref{fig:sundials-lambda}.
One might expect that this automatically improves the output of the decoder applied to linear interpolations in latent space, as depicted 
in \cref{fig:cows-qualitative}. However, in the example of \cref{fig:arcs-summary}, higher values of $\lambda$ do not come with an improvement in the interpolation quality. This is also reflected quantitatively in \cref{fig:quantitative}. In fact, for more complicated latent manifolds, various factors might additionally impact the interpolation error, such as the sign of the intrinsic curvature, an increased tangential distortion of the encoder map for flatter latent manifolds, or the interplay of the differently trained encoder and decoder map.\\[0.5em]

\noindent\textbf{Data availability statement}\\[0.5em]
The code used in this paper is available online in a Gitlab repository \href{https://gitlab.com/jubrau/LBD}{gitlab.com/jubrau/LBD} and in the datastore of the University of Münster \href{https://doi.org/10.17879/48978454901}{doi:10.17879/48978454901} \cite{Br23}.\\[0.5em]

\noindent\textbf{Acknowledgements}\\[0.5em] This work was supported by the Deutsche Forschungsgemeinschaft (DFG, German Research Foundation) 
via project 211504053 - Collaborative Research Center 1060 and via Germany's Excellence Strategy project 390685813 - 
Hausdorff Center for Mathematics and  project 390685587 - Mathematics M\"unster: Dynamics-Geometry-Structure.
%%-----------------------------
%%      bibliography
%%-----------------------------
\bibliographystyle{acm}
\bibliography{references}
\end{document}